\documentclass[12pt,a4paper]{article}
\usepackage{amsmath,amsthm,amssymb}
\usepackage[utf8]{inputenc}
\usepackage[T1]{fontenc}
\usepackage{enumerate}
\usepackage{geometry}
\usepackage{comment}
\usepackage{bbm}
\usepackage{todonotes}
\usepackage{mathrsfs}
\usepackage{mathabx}
\usepackage{nicefrac}

\usepackage{microtype}

\providecommand{\noopsort}[1]{} 
\usepackage[sort&compress,numbers]{natbib}

\usepackage{color}

\PassOptionsToPackage{hyphens}{url}
\usepackage{hyperref}

\usepackage{cleveref}
\crefname{Theorem}{Theorem}{Theorems}
\crefname{Th}{Theorem}{Theorems}
\crefname{Prop}{Proposition}{Propositions}
\crefname{Lemma}{Lemma}{Lemmas}
\crefname{Cor}{Corollary}{Corollaries}
\crefname{Thx}{Theorem}{Theorems}
\crefname{Remark}{Remark}{Remarks}
\crefname{Def}{Definition}{Definitions}
\crefname{Example}{Example}{Examples}
\crefname{Question}{Problem}{Problems}
\crefname{section}{Section}{Sections}

\newtheorem{Th}{Theorem}[section]
\newtheorem{Prop}[Th]{Proposition}
\newtheorem{Lemma}[Th]{Lemma}
\newtheorem{Cor}[Th]{Corollary}
\newtheorem{Thx}{Theorem}
\theoremstyle{definition}
\newtheorem{Remark}[Th]{Remark}
\newtheorem{Def}{Definition}[section]

\newcommand{\gp}{\mathfrak{p}}
\newcommand{\q}{\mathbb{Q}}

\newcommand{\beq}{\begin{equation}}
\newcommand{\eeq}{\end{equation}}

\def\scalar(#1,#2){(#1\mid#2)}

\newcommand{\vol}{\operatorname{vol}}

\newcommand{\co}{\mathcal{O}}

\newcommand{\gb}{\mathfrak{b}}
\newcommand{\xbm}{(X,{\cal B},\mu)}

\newcommand{\Q}{\mathbb{Q}}
\newcommand{\R}{{\mathbb{R}}}

\newcommand{\Z}{{\mathbb{Z}}}
\newcommand{\N}{{\mathbb{N}}}
\newcommand{\PP}{{\mathbb{P}}}

\newcommand{\vep}{\varepsilon}

\newcommand{\mob}{\boldsymbol{\mu}}

\newcommand{\raz}{\mathbbm{1}}

\newcommand{\OK}{\mathcal{O}_K}

\title{$\mathfrak{B}$-free integers in number fields and dynamics}
\author{Francisco Ara\'{u}jo \hspace{1cm}
Aurelia Dymek
\and Joanna Ku\l{}aga-Przymus}

\begin{document}
\bibliographystyle{siam}
\maketitle

\abstract{In 2010, Sarnak initiated the study of the dynamics of the system determined by the square of the M\"obius function (the characteristic function of the square-free integers). We deal with his program in the more general context of $\mathfrak{B}$-free integers in number fields, suggested 5 years later by Baake and Huck. This setting encompasses the classical square-free case and its generalizations. Given a number field $K$, let $\mathfrak{B}$ be a family of pairwise coprime ideals in its ring of integers $\mathcal{O}_K$, such that $\sum_{\mathfrak{b}\in\mathfrak{B}}1/|\OK / \mathfrak{b}|<\infty$. We study the dynamical system determined by  the set $\mathcal{F}_\mathfrak{B}=\OK\setminus \bigcup_{\mathfrak{b}\in\mathfrak{B}}\mathfrak{b}$ of $\mathfrak{B}$-free integers in $\OK$. We show that the characteristic function $\raz_{\mathcal{F}_\mathfrak{B}}$ of $\mathcal{F}_\mathfrak{B}$ is generic along the natural F\o{}lner sequence for a probability measure on $\{0,1\}^{\OK}$, invariant under the multidimensional shift. The corresponding measure-theoretical dynamical system is proved to be isomorphic to an ergodic rotation on a compact Abelian group. In particular, it is of zero Kolmogorov entropy. Moreover, we provide a description of ``patterns'' appearing in $\mathcal{F}_\mathfrak{B}$ and compute the topological entropy of the orbit closure of $\raz_{\mathcal{F}_\mathfrak{B}}$. Finally, we show that this topological dynamical system 
has a non-trivial topological joining with an ergodic rotation on a compact Abelian group. 
}

\tableofcontents
\section{Introduction}
\subsection{Motivation}\label{moti}
The M\"obius function $\mob$ is one of central objects in number theory. Recall that it is given by $\mob(1)=1$, $\mob(n)=(-1)^m$ when $n$ is a product of $m$ distinct primes, and takes value zero for $n$ which are not square-free (i.e.\ are divisible by the square of a prime). The function $\mob$ displays quite a random behavior reflected in the bound $\sum_{n\leq N}\mob(n)=\rm{o}(N)$, equivalent to the Prime Number Theorem (cf.\ \cite{MR0434929}, p.\ 91). Moreover, $\sum_{n\leq N}\mob(n)={\rm O}_\vep(N^{1/2+\vep})$ (for each $\vep>0$) is already equivalent to the Riemann hypothesis~\cite{MR882550}. More recently, $\mob$ has become of an interest also from the point of view of ergodic theory. Sarnak, in his seminal paper~\cite{sarnak-lectures} conjectured that
\begin{equation}\label{eq:sarnak}
\sum_{n\leq N}\mob(n) f(T^nx) ={\rm o}(N)
\end{equation}
for every zero topological entropy homeomorphism $T$ of a compact metric space $X$, every $f\in C(X)$ and every $x\in X$. The convergence resulting from~\eqref{eq:sarnak} follows from the Chowla conjecture from the 1960's~\cite{MR0177943} on higher order self-correlations of $\mob$~\cite{MR3014544,MR3622068}. Sarnak also proposed to study the dynamical systems related to~$\mob$ and to $\mob^2$. The latter is the subshift determined by the characteristic function of square-free integers. In each case, we extend the corresponding sequence (either $\mob$ or $\mob^2$) symmetrically and we study its orbit closure in the space $A^\Z$ (with $A=\{-1,0,1\}$ or $\{0,1\}$) of sequences under the left shift~$S$ (for $z\in A^\Z$, the corresponding orbit closure will be denoted by $X_z$). While the dynamics of $(X_{\mob}, S)$ is still quite mysterious, we can study it through the much simpler system $(X_{\mob^2},S)$ which is a topological factor of $(X_{\mob},S)$ via the map ${(x_n)}_{n\in\Z}\mapsto {(x_n^2)}_{n\in\Z}$. Sarnak~\cite{sarnak-lectures} announced several results concering $(X_{\mob^2},S)$:
\begin{enumerate}[(A)]\label{strony}
\item sequence $\mob^2$ is a generic point for a natural $S$-invariant zero Kolmogorov entropy probability measure $\nu_{\mob^2}$ on $\{0,1\}^\Z$,\label{s:A}
\item the topological entropy of $(X_{\mob^2},S)$ is equal to $6/\pi^2$,\label{s:B}
\item subshift $X_{\mob^2}$ consists  of so-called \emph{admissible sequences}, i.e.\ $x\in X_{\mob^2}$ is equivalent to $|\text{supp }x \bmod p^2|<p^2$ for each $p\in\mathcal{P}$,\footnote{We denote by $\text{supp }x$ the support of $x$, i.e.\ $\text{supp }x=\{n\in\Z : x(n)\neq 0\}$ and $\mathcal{P}$ stands for the set of primes.}\label{s:C}
\item dynamical system $(X_{\mob^2},S)$ is proximal and $\{(\dots,0,0,0,\dots)\}$ is the unique minimal subset of $X_{\mob^2}$,\label{s:D}
\item the maximal equicontinuous factor of $(X_{\mob^2},S)$ is trivial but $(X_{\mob^2},S)$ has a non-trivial joining with a rotation on the compact Abelian group $G=\prod_{p\in\mathcal{P}}\Z/p^2\Z$.\label{s:E}
\end{enumerate}
The above statements have been discussed by several authors, both in the setting proposed by Sarnak, and for some natural generalizations \cite{MR3055764,MR3430278,Huck:2014aa,MR3428961,MR3070541,MR3296562,MR3485709}. The purpose of this paper is to extend these works, providing a unified approach in all previous cases which we list here.

\paragraph{Square-free integers}
Sarnak's program was first addressed by Cellarosi and Sinai, who cover~\eqref{s:A} in~\cite{MR3055764}: they showed that $\mob^2$ is generic for a shift-invariant measure $\nu_{\mob^2}$ on $\{0,1\}^\Z$, and that $(X_{\mob^2},S,\nu_{\mob^2})$ is isomorphic to a rotation on the compact Abelian group $\prod_{p\in\mathcal{P}}\Z/p^2\Z$ (in particular, $(X_{\mob^2},S,\nu_{\mob^2})$ is of zero Kolmogorov entropy). Statements~\eqref{s:B} and~\eqref{s:C} were proved in~\cite{MR3430278} by Peckner who also showed that $(X_{\mob^2},S)$ has only one measure of maximal entropy, i.e.\ it is intrinsically ergodic. The proofs of~\eqref{s:D} and~\eqref{s:E} were provided later by Huck and Baake~\cite{Huck:2014aa}, in the more general setting of $k$-free lattice points, see below.

\paragraph{$\mathscr{B}$-free integers}
The set of square-free numbers is a special case of a set of integers with no factors in a given set $\mathscr{B}\subseteq\N\setminus\{1\}$, which is called the $\mathscr{B}$-free set and denoted by $\mathcal{F}_\mathscr{B}$:
\begin{equation}\label{multiples}
\mathcal{F}_\mathscr{B}=\Z \setminus \mathcal{M}_\mathscr{B},\ \text{where }\mathcal{M}_\mathscr{B}=\bigcup_{b\in\mathscr{B}}b\Z.
\end{equation}
Such sets $\mathcal{M}_\mathscr{B}$ were studied already in 1930's by numerous mathematicians, including Behrend, Chowla, Davenport, Erd\H{o}s and Schur, see~\cite{MR1414678}. Clearly, $\mathscr{B}=\{p^2\in\N : p \text{ is prime} \}$ yields $\raz_{\mathcal{F}_\mathscr{B}}=\mob^2$.

In the general setting~\eqref{multiples}, it is unclear how to define a reasonable analogue of $\mob$. However, we can put $\eta:=\raz_{\mathcal{F}_\mathscr{B}}$ and study the associated dynamical system $(X_\eta,S)$. The asymptotic density of $\mathcal{M}_\mathscr{B}$ (and $\mathcal{F}_\mathscr{B}$) exists only under additional assumptions on $\mathscr{B}$~\cite{MR1512943,MR0026088,arxiv_density}. In particular, this happens whenever
\begin{equation}\label{eq:condB}\tag{I}
\mathscr{B} \text{ is infinite, pairwise coprime and }\sum_{b\in\mathscr{B}}{1}/{b}<\infty,
\end{equation}
(which clearly includes the classical square-free case). In setting~\eqref{eq:condB}, statements \eqref{s:A}-\eqref{s:E} were studied by Abdalaoui, Lema\'nczyk and de la Rue~\cite{MR3428961}. In particular, they proved \eqref{s:A}-\eqref{s:C} (with $\mob^2$ replaced with~$\eta$). The intrinsic ergodicity in this context was proved in~\cite{MR3356811}, which also contains a full description of the set of invariant probability measures for $(X_\eta,S)$. 

\begin{Remark}
For general sets $\mathscr{B}\subseteq \N\setminus \{1\}$, not necessarily satisfying \eqref{eq:condB},  sequence $\eta$ is generic if and only if $\mathscr{B}$ is Besicovitch (i.e.\ the asymptotic density of $\mathcal{F}_\mathscr{B}$ exists) \cite[Proposition~E]{MR3803141}. The formula for the topological entropy of $\mathscr{B}$-free subshifts was determined in \cite[Proposition~K]{MR3803141} in the hereditary case (see also \cite[Theorem~F]{MR4802750} for a more general setting). Moreover, a $\mathscr{B}$-free subshift of positive topological entropy consists of all $\mathscr{B}$-admissible sequences if and only if $\mathscr{B}$ satisfies \eqref{eq:condB} \cite[Corollary 1.1]{MR4665256}. The proximality of a $\mathscr{B}$-free subshift is characterized in \cite[Theorem~B]{MR3803141}. The maximal equicontinuous factor of a $\mathscr{B}$-free subshift is described in \cite[Theorem D]{MR3947636}. The existence of a non-trivial topological joining as in \eqref{s:E} is shown in \cite[Proposition 3.21]{MR3803141}.
\end{Remark}

\paragraph{$k$-free lattice points}
Another way of generalizing the square-free setting considered by Sarnak was discussed by Pleasants and Huck~\cite{MR3070541}. Given a lattice $\Lambda$ in $\R^d$ (i.e.\ a discrete cocompact subgroup), they defined the set of $k$-free lattice points by 
\begin{equation}\tag{II}\label{dwojka}
\mathcal{F}_k=\mathcal{F}_k(\Lambda):=\Lambda \setminus \bigcup_{p\in\mathcal{P}} p^k \Lambda.
\end{equation}
The associated dynamical system is the orbit closure $X_k$ of $\raz_{\mathcal{F}_k}\in\{0,1\}^\Lambda$ under the corresponding multidimensional shift ${(S_\lambda)}_{\lambda\in\Lambda}$ (see Theorem~\ref{szift}). It was shown in~\cite{MR3070541} that $\raz_{\mathcal{F}_k}$ is generic along the F\o{}lner sequence $([-n,n]^d\cap \Lambda)$ under ${(S_\lambda)}_{\lambda\in\Lambda}$ for a probability measure $\nu$ on $X_k$, and that $(X_k,{(S_\lambda)}_{\lambda\in\Lambda}, \nu)$ is measure-theoretically isomorphic to a rotation on the compact Abelian group $\prod_{p\in\mathcal{P}}\Lambda/p^k\Lambda$ (cf.\ \eqref{s:A} above). A formula for the topological entropy of $(X_k,{(S_\lambda)}_{\lambda\in\Lambda})$ was also given (cf.\ \eqref{s:B} above). Finally, $X_k$ was described in terms of admissible patches (cf.\ \eqref{s:C} above).

\paragraph{$k$-free integers in number fields}
Cellarosi and Vinogradov \cite{MR3296562} discussed the setting of algebraic number fields and defined all the aforementioned objects in this context. For a finite extension $K$ of $\Q$, they studied the following subset of the ring of integers $\mathcal{O}_K\subseteq K$:
\begin{equation}\label{trojka}\tag{III}
\mathcal{F}_k=\mathcal{F}_k(\mathcal{O}_K):=\mathcal{O}_K \setminus \bigcup_{\mathfrak{p}\in\mathfrak{P}}\mathfrak{p}^k
\end{equation}
($\mathfrak{P}$ stands for the family of all prime ideals $\mathfrak{p}$ in $\mathcal{O}_K$ and $\mathfrak{p}^k$ stands for $\underbrace{\mathfrak{p}\dots\mathfrak{p}}_k$). They considered the orbit closure $X_k$ of $\raz_{\mathcal{F}_k}\in \{0,1\}^{\mathcal{O}_K}$ under the multidimensional shift ${(S_a)}_{a\in\mathcal{O}_K}$. Using similar methods as developed earlier by Cellarosi and Sinai in~\cite{MR3055764}, they proved an analogue of~\eqref{s:A}. In fact, they showed that $\raz_{\mathcal{F}_k}$ is generic for an ergodic probability measure $\nu$ on $\{0,1\}^{\mathcal{O}_K}$ along the image of the natural F\o{}lner sequence $([-n,n]^d)$ via the Minkowski embedding, and that $(X_k,{(S_a)}_{a\in \mathcal{O}_K}, \nu)$ is measure-theoretically isomorphic to a rotation on the compact Abelian group $\prod_{\mathfrak{p}\in\mathfrak{P}}\mathcal{O}_K/\mathfrak{p}^k$.

\paragraph{$\mathscr{B}$-free lattice points} 
Baake and Huck in their survey~\cite{MR3485709} extended~\eqref{dwojka} and defined $\mathscr{B}$-free lattice points in a lattice $\Lambda\subseteq \R^d$:
\begin{equation}\tag{IV}\label{czworka}
\mathcal{F}_{\mathscr{B}}=\mathcal{F}_{\mathscr{B}}(\Lambda):=\Lambda \setminus \bigcup_{b\in\mathscr{B}}b\Lambda,
\end{equation}
where $\mathscr{B}\subseteq \N\setminus\{1\}$ is an infinite pairwise coprime set with $\sum_{b\in\mathscr{B}}1/{b^d}<\infty$. They announced similar results as for $k$-free lattice points, leaving the details to the reader due to the similarity of methods.

\paragraph{$\mathfrak{B}$-free integers in number fields}
The setting we deal with in this paper also origins from~\cite{MR3485709}, where $\mathfrak{B}$-free integers in number fields are defined. Given a finite extension $K$ of $\Q$, with the ring of integers $\OK$, we set
\begin{equation}\tag{V}\label{nasz:setting}
\mathcal{F}_{\mathfrak{B}}=\mathcal{F}_{\mathfrak{B}}(\mathcal{O}_K):=\mathcal{O}_K\setminus \bigcup_{\mathfrak{b}\in\mathfrak{B}}\mathfrak{b},
\end{equation}
where $\mathfrak{B}$ is an infinite  pairwise coprime collection of ideals in $\mathcal{O}_K$ with 
$\sum_{\mathfrak{b}\in\mathfrak{B}}{1}/{|\mathcal{O}_K/\mathfrak{b}|}<\infty$. Under these assumptions we will say that $\mathfrak{B}$ is Erd\H{o}s.\footnote{This is consistent with the nomenclature from the one-dimensional case.}

\subsection{Results}
The shortest possible way to state the main results is the following:
\begin{center}
\eqref{s:A}-\eqref{s:E} are true in setting \eqref{nasz:setting}.
\end{center}
A more detailed formulation, together with the proofs, can be found in \cref{se:results}. Moreover, in \cref{a:a}, we show that~\eqref{nasz:setting} encompasses all the other cases listed above in \cref{moti}, i.e.~\eqref{eq:condB}-\eqref{czworka}. Theorem \ref{s:D} is proved in a more general setting in \cite[Theorem 1.2]{MR4251829}, where the proximality of $(X_\mathfrak{B},{(S_a)}_{a\in\OK})$ is characterized for general $\mathfrak{B}$.

\section{Basic objects, definitions, notation}
\paragraph{Number fields and ideals}\label{S:21}

Let $K$ be an algebraic number field of degree $d=[K:\Q]$ with integer ring $\OK$. It is well known
(see for example Theorem 51 in the Appendix~B of~\cite{MR3822326}) that for every K there is some $\alpha\in K$ such that $K=\mathbb{Q}[\alpha]$. As in every Dedekind domain, all proper non-zero ideals in $\OK$ factor (uniquely, up to the order) into a product of prime ideals.
We will denote ideals in $\mathcal{O}_K$ by $\mathfrak{a},\mathfrak{b},\dots$  We have
\[
\mathfrak{a}+\mathfrak{b}=\{a+b : a\in\mathfrak{a},b\in\mathfrak{b}\},\ \mathfrak{a}\mathfrak{b}=\{a_1b_1+\dots+a_kb_k : a_i\in \mathfrak{a},b_i\in\mathfrak{b},1\leq i\leq k\}.
\]
We say that an ideal $\mathfrak{b}$ divides an ideal $\mathfrak{a}$ if there exists an ideal $\mathfrak{c}$ with $\mathfrak{a}=\mathfrak{b}\mathfrak{c}$. Equivalently, $\mathfrak{a}\subseteq \mathfrak{b}$.
Proper ideals $\mathfrak{a},\mathfrak{b}$ are said to be \emph{coprime} whenever $\mathfrak{a}+\mathfrak{b}=\mathcal{O}_K$. Equivalently, $\mathfrak{a},\mathfrak{b}$ do not share factors: there are no non-trivial ideals $\mathfrak{a}',\mathfrak{b}',\mathfrak{c}$ such that $\mathfrak{a}=\mathfrak{c}\mathfrak{a}'$ and $\mathfrak{b}=\mathfrak{c}\mathfrak{b}'$. If $\mathfrak{a}$ and $\mathfrak{b}$ are coprime, we have $\mathfrak{a}\mathfrak{b}=\mathfrak{a}\cap\mathfrak{b}$. The \emph{algebraic norm} of an ideal $\mathfrak{a}\neq \{0\}$ is defined as $N(\mathfrak{a}):=|\mathcal{O}_K/\mathfrak{a}|=[\mathcal{O}_K:\mathfrak{a}]$. The \emph{Dedekind zeta function} is given by
\begin{equation}\label{eq:zetadede}
\zeta_K(s)=\sum_{\mathfrak{a}\neq \{0\}}\frac{1}{N(\mathfrak{a})^s}=\prod_{\mathfrak{p}\in\mathfrak{P}}\left(1-\frac{1}{N(\mathfrak{p})^s} \right)^{-1} \text{ for }s\text{ with }\Re(s)>1. 
\end{equation}
We also have the Prime Ideal Theorem over any number field, as proven by Landau in the second part of \cite{MR1511191}.
	
	\begin{Th}
		\label{thm: Prime Ideal Theorem}
		Let $K$ be a number field with ring of integers $\co_K$. Denoting by $\pi_K(X)$ the number of prime ideals $\gp$ of $\co_K$ such that $N(\gp) \leq X$, we have $$\lim_{X \rightarrow \infty} \frac{\pi_K(X)}{X/\log(X)} = 1.$$
	\end{Th}

	 For a number field $K$ of degree $d$, there are $d$ distinct embeddings of $K$ into $\mathbb{C}$. To be more presice, let $\alpha$ be such that $K = \q[\alpha]$ and let $f$ be the minimal polynomial of $\alpha$. Then for each of the $d$ distinct roots $\theta$ of the polynomial  $f$, there is a unique (injective) field homomorphism $\phi\colon K\to \mathbb{C}$ such that $\phi(\alpha)=\theta$. Moreover, all injective field homomorphisms are of this form. This allows us to define the Minkowski embedding $\sigma\colon \co_K \rightarrow \mathbb{C}^d$ in the following way: \[\sigma(x):= (\phi(x))_{\phi \in \text{Hom}_\q(K, \mathbb{C})},\]
    where $\text{Hom}_\q(K, \mathbb{C})$ denotes the set of all injective homomorphisms $\phi\colon K\to\mathbb{C}$ over $\q$.
    Clearly, the above object is uniquely defined only up to the permutation of the coordinates. Moreover, the image of $\OK$ via $\sigma$ in $\mathbb{C}^d$ is a lattice.  On $\OK$, we will use the norm inherited from the supremum norm of the Minkowski embedding, that is
     \[\|x\| := \|\sigma(x)\|_{\text{sup}}=\sup_{\phi \in \text{Hom}_\q(K,\mathbb{C})}|\phi(x)|.\]
    Given a lattice $\Lambda\subseteq \mathbb{R}^d$, one defines so-called \emph{successive minima} in the following way:
    \[
    \lambda_i(\Lambda):=\inf\{r\geq 0 : B(0,r)\cap \Lambda \text{ contains $i$ linearly independent vectors} \},
    \]
    where $B(0,r)=\{x\in\mathbb{R}^d : \|x\|\leq r\}$ and $\|\cdot\|$ denotes the supremum norm on $\mathbb{R}^d$. Now, using the Minkowski embedding, one can easily transfer this notion to lattices $\Gamma\subseteq\OK$ by setting
    \(
    \lambda_i(\Gamma):=\lambda_i(\sigma(\Gamma)).
    \)
    We will write $\lambda_1(\Gamma) \asymp_K \lambda_d(\Gamma)$ to mean that there are constants $c_K$ and $C_K$, depending only on $K$ such that $c_K\lambda_1(\Gamma) \leq \lambda_d(\Gamma) \leq C_K\lambda_1(\Gamma)$, independently of the ideal $\Gamma\subset \co_K$. By writing $\lambda_1(\gb)\asymp_K \lambda_d(\gb)$ we will mean that the above holds for every lattice being an ideal in $\co_K$.
    We have the following result.  
	\begin{Th}[{\cite[Corollary 4]{Fraczyk}}]
		\label{lm: lambda_i asymp n(gb)(1/n)}
		Let $K$ be a number field of degree $d$. Then for any ideal $\gb$ of $\co_K$, we have $$\lambda_1(\gb) \asymp_K \lambda_d(\gb) \asymp_K N(\gb)^{\nicefrac{1}{d}}.$$ 
	\end{Th}
	
	Notice that while \Cref{lm: lambda_i asymp n(gb)(1/n)} doesn't hold for lattices in general, Minkowski's second theorem (see Theorem 2E in \cite{MR1176315}) shows that \begin{equation}
		\label{eq: product of Lambdas}
		[\co_K:\Gamma] \asymp_K \lambda_1(\Gamma) \dots \lambda_d(\Gamma).
	\end{equation} 

We refer the reader to \cite{MR0195803,MR1697859} for more background information on algebraic number theory.

\paragraph{$\mathfrak{B}$-free integers in number fields}
Let $\mathfrak{B}:=\{\mathfrak{b}_\ell : \ell\geq 1\}$ be a collection of ideals in the integer ring $\mathcal{O}_K$ of an algebraic number field $K$. E.g.\ we can take $\mathfrak{B}=\{\mathfrak{p}^k: \mathfrak{p}\in\mathfrak{P}\}$, $k\geq 2$ (recall that $N(\mathfrak{a}\mathfrak{b})=N(\mathfrak{a})N(\mathfrak{b})$ for any ideals $\mathfrak{a},\mathfrak{b}$ and cf.\ \eqref{eq:zetadede}).
\begin{Def}
We say that 
\begin{enumerate}[(i)]
\item $\mathfrak{a}$ is \emph{$\mathfrak{B}$-free} whenever $\mathfrak{a}\not\subseteq \mathfrak{b}_\ell$ for all $\ell\geq 1$;
\item
$a\in\mathcal{O}_K$ is \emph{$\mathfrak{B}$-free} if the principal ideal $(a):=a\mathcal{O}_K$ is $\mathfrak{B}$-free.
\end{enumerate}
\end{Def}
We denote the set of $\mathfrak{B}$-free integers in $\mathcal{O}_K$ by $\mathcal{F}_\mathfrak{B}$. 
\begin{Remark}
Since for any ideal $\mathfrak{b}\subseteq \OK$ and $a\in \mathcal{O}_K$ we have $a\not \in \mathfrak{b}$ if and only if $ (a) \not\subseteq \mathfrak{b}$, it follows immediately that
\begin{equation}\label{Z:1}
\mathcal{F}_\mathfrak{B}=\mathcal{O}_K\setminus \bigcup_{\ell\geq 1}\mathfrak{b}_\ell.
\end{equation}
\end{Remark}
The characteristic function of $\mathcal{F}_\mathfrak{B}$ will be denoted by $\eta\in \{0,1\}^{\OK}$, i.e.\
\begin{equation}\label{Z:2}
\eta(a)=\begin{cases}
1,& \text{if }a \text{ is }\mathfrak{B}\text{-free},\\
0,& \text{otherwise}.
\end{cases}
\end{equation}

\paragraph{F\o{}lner sequences}
Let $\mathbb{G}$ be a countable group. 
\begin{Def}[\cite{MR0079220}]
We say that ${(F_n)}_{n\geq 1}\subseteq \mathbb{G}$ is a \emph{F\o{}lner sequence} in $\mathbb{G}$ if $\bigcup_{n\geq 1}F_n=\mathbb{G}$ and
$$
\lim_{n\to \infty}\frac{|gF_n\cap F_n|}{|F_n|}=1
$$
for each $g\in \mathbb{G}$. If $F_n\subseteq F_{n+1}$ for each $n\geq 1$, we say that $(F_n)_{n\geq 1}$ is \emph{nested}.
\end{Def}
\begin{Def}[\cite{Shu88}]
A sequence of finite sets ${(F_n)}_{n\geq 1}\subseteq\mathbb{G}$ is said to be \emph{tempered} if, for some $C>0$ and all $n\in\N$,
$$
\left|\bigcup_{k<n}F_k^{-1}F_n\right|\leq C |F_n|.
$$
\end{Def}
For $\mathbb{G}=\OK$, the usual F\o lner sequence is given by $B_n:=\sigma^{-1}(B(0,n))$, where $\sigma$ is the Minkowski embedding. This is a tempered F\o lner sequence.
\begin{Def}
Given a set $A \subset \mathbb{G}$ and a F\o lner sequence $(F_n)_{n\geq1}$, we define the \emph{upper} and \emph{lower densities of $A$ along} $(F_n)_{n\geq1}$:
\[
\overline{d}_{(F_n)}(A):=\limsup _{n \rightarrow \infty} \frac{\left|A \cap F_n\right|}{\left|F_n\right|} \text { and } \underline{d}_{(F_n)}(A):=\liminf _{n \rightarrow \infty} \frac{\left|A \cap F_n\right|}{\left|F_n\right|}.
\]
If these agree, we write the limit as $d_{(F_n)}(A)$, which we call the \emph{density of $A$ along $(F_n)$}.
When $\mathbb{G}=\mathcal{O}_K$ and $F_n=B_n$, we simply write $\overline{d}(A), \underline{d}(A), d(A)$ for each corresponding density.
\end{Def}
\begin{Def}\label{besicovitch}
We say that $\mathfrak{B}=\{\mathfrak{b}_{\ell}\}_{\ell\geq1}$ is \emph{Besicovitch} if $d(\mathcal{F}_\mathfrak{B})$ exists.
\end{Def}

\paragraph{Basic notions from dynamics}
Throughout this paper, we deal with groups \( \mathbb{G} \) isomorphic to \( \mathbb{Z}^d \) for some \( d \ge 1 \). Accordingly, we recall the dynamical notions only in this setting.

A \emph{topological dynamical system} is a pair $(X,(T_g)_{g\in \mathbb{G}})$, where $X$ is a compact metrizable space and $\mathbb{G}$ is acting on $X$ by homeomorphisms $T_g$, $g\in \mathbb{G}$. Given $y\in X$, we will denote by $X_y$ its orbit closure (the action will be always clear from the context), i.e.\ $X_y=\overline{\{T_gy : g\in\mathbb{G}\}}$. An action ${(T_g)}_{g\in \mathbb{G}}$ on $X$ is called \emph{transitive} if $X=X_y$ for some $y\in X$. A topological dynamical system $(X,{(T_g)}_{g\in \mathbb{G}})$ is called \emph{proximal} if, for all $x,y\in X$,
$$
\liminf_{g\to\infty}d(T_gx,T_gy)=0.
$$
If $(T_g)_{g\in \mathbb{G}}$ and $(S_g)_{g\in\mathbb{G}}$ act respectively on $X$ and $Y$, we say that $A\subseteq X\times Y$ is their \emph{topological joining} whenever $A$ is closed, invariant under ${(T_g\times S_g)}_{g\in\mathbb{G}}$ and has full projections on both coordinates. We say that the joining $A$ is non-trivial if $A\neq X\times Y$.

The primary example of a topological dynamical system is a \emph{subshift}, i.e.\ a closed subset $X\subseteq \mathcal{A}^\mathbb{G}$ (where $\mathcal{A}$ is a finite set called the alphabet) which is invariant under the action of $\mathbb{G}$ by commuting translations:
\begin{equation}\label{eq:translations}
S_g((x_h)_{h\in\mathbb{G}})=(x_{h+g})_{h\in\mathbb{G}},\ g\in\mathbb{G}.
\end{equation}
In this paper, we have $\mathcal{A}=\{0,1\}$.

 Let $M(X,(T_g)_{g\in \mathbb{G}})$ stand for the set of $(T_g)_{g\in G}$-invariant Borel probability measures on~$X$. Each choice of $\nu \in M(X,(T_g)_{g\in \mathbb{G}})$ gives rise to a \emph{measure--preserving dynamical system} $(X,\mathcal B,\nu,(T_g)_{g\in \mathbb{G}})$, where $\mathcal B$ denotes the Borel $\sigma$--algebra. A measure $\nu \in M(X,(T_g)_{g\in \mathbb{G}})$ is called \emph{ergodic} if we have $\nu(A\triangle T_g^{-1}A)=0$ for any $g\in\mathbb{G}$ only for $A\in\mathcal B$ such that $\nu(A)=0$ or $\nu(X\setminus A)=0$. A measure--preserving dynamical system $(X,\mathcal B,\nu,(T_g)_{g\in \mathbb{G}})$ is called \emph{ergodic} if $\nu$ is ergodic.

We say that $x\in X$ is \emph{generic} for $\nu$ (under ${(T_g)}_{g\in\mathbb{G}}$) along F\o{}lner sequence ${(F_n)}_{n\geq 1}\subseteq \mathbb{G}$ if
\begin{equation}\label{generuje2}
\frac{1}{|F_n|}\sum_{g\in F_n}f(T_gx) \to \int_X f\ d\nu
\end{equation}
for any $f\in C(X)$. 
\begin{Remark}\label{uw:geny}
In case of subshifts it suffices to check~\eqref{generuje2} for a certain ``easy'' family of functions $f$ to obtain that $x$ is a generic point (along a F\o lner sequence). Namely, for finite disjoint sets $A,B\subseteq \mathbb{G}$, let
\begin{equation}\label{generuje1}
C_{A,B}:=\{x\in \{0,1\}^\mathbb{G} : x(a)=1 \text{ for }a\in A \text{ and }x(b)=0\text{ for }b\in B\}
\end{equation}
be the corresponding \emph{cylinder set}. We write $C_A^1$ for $C_{A,\emptyset}$ and $C_B^0$ for $C_{\emptyset,B}$. Since locally constant functions span a dense subalgebra of $C(X)$, we obtain (using the inclusion-exclusion principle) that it suffices to check~\eqref{generuje2} for functions of the form $\raz_{C_B^0}$ for finite $B\subseteq \mathbb{G}$.
\end{Remark}
\begin{Th}[Pointwise Ergodic Theorem]\label{tw:point}
Let $\mathbb{G}$ be isomorphic to $\mathbb{Z}^d$.\footnote{In \cite{MR1865397}, \cref{tw:point} is proved in the more general case of discrete amenable groups, see also~\cite{MR2052281} and the earlier works \cite{MR1546100,MR0055415,MR0354926}.} Let $\nu\in M(X,{(T_g)}_{g\in\mathbb{G}})$ be ergodic and let $f\in L^1(X,\nu)$. Then, for $\nu$-a.e.\ $x\in X$,~\eqref{generuje2} holds for any tempered F\o{}lner sequence ${(F_n)}_{n\geq 1}$.
\end{Th}

\begin{Remark}\label{uw:uniqergo}
If $(X,{(T_g)}_{g\in\mathbb{G}})$ is uniquely ergodic (i.e.\ $|M(X,(T_g)_{g\in\mathbb{G}})|=1$) then~\eqref{generuje2} holds for every continuous function $f$, at every point $x$, along every F\o{}lner sequence ${(F_n)}_{n\geq 1}$. The proof goes along the same lines as in the classical case of $\Z$-actions, cf.\ \cite{MR648108}. (Since every F\o{}lner sequence has a tempered subsequence, as shown in~\cite{MR1865397}, we can drop the restriction that ${(F_n)}_{n\geq 1}$ is tempered, present in \cref{tw:point}.)
\end{Remark}



Given a topological dynamical system $(X,(T_g)_{g\in\mathbb{G}})$, we will denote by $h_{top}(X,(T_g)_{g\in\mathbb{G}})$ its \emph{topological entropy}, see~\cite{MR0417391,MR0453976} for the definition. In case of a subshift $X\subseteq \{0,1\}^\mathbb{G}$, we have the following:
\begin{equation}\label{uw:bloki}
h_{top}(X,(S_g)_{g\in\mathbb{G}})=\lim_{n\to\infty}\frac{1}{|F_n|}\log_2 \gamma(n),
\end{equation}
where $(F_n)$ is an arbitrary F\o lner sequence and 
\[
\gamma(n)=|\{A\in \{0,1\}^{F_n}: x_{g+h}=A_g \text{ for some }x\in X \text{ and }h\in \mathbb{G}, \text{ and all }g\in F_n \}|.\footnote{The proof goes by the same token as for $\Z$-actions, cf.\ Corollary 14.7 in~\cite{MR1958753}.}
\]
For $\nu\in M(X,{(T_g)}_{g\in\mathbb{G}})$, we denote by $h(X,{(T_g)}_{g\in\mathbb{G}}, \nu)$ the corresponding \emph{measure-theoretic entropy}, see~\cite{MR0335754,MR0316680,MR910005} for the definition. For any $\nu\in M(X,(T_g)_{g\in\mathbb{G}})$, we have
\begin{equation}\label{uw:wypu}
h(X,(T_g)_{g\in\mathbb{G}},\nu)=\int h(X,(T_g)_{g\in\mathbb{G}},\nu_y)\ dQ(y),
\end{equation}
where $\int \nu_y \ dQ(y)$ is the ergodic decomposition of $\nu$. Moreover, there is the following relation between measure-theoretic and topological entropy, known as the variational principle:\footnote{For the first time the variational principle was proved in~\cite{MR0417391} under some restrictions. See also~\cite{MR0437716} for the variational principle for topological pressure and~\cite{MR1110315} for the variational principle for entropy of $\R^d$-actions.}
\[
h_{top}(X,(T_g)_{g\in\mathbb{G}})=\sup_{\nu\in M(X,(T_g)_{g\in\mathbb{G}})}h(X,(T_g)_{g\in \mathbb{G}}, \nu).
\]
Every subshift over a finite alphabet has at least one measure of maximal entropy~\cite{MR0444904}.

\paragraph{Dynamical system outputting $\mathfrak{B}$-free integers}\label{szift}
Consider the product of finite groups $\OK/{\mathfrak{b}_\ell}$
\begin{equation}\label{eq:grupa}
G:=\prod_{\ell\geq 1}\mathcal{O}_K/{\mathfrak{b}_\ell}
\end{equation}
with coordinatewise addition. The Haar measure $\PP$ on $G$ is the product of the corresponding counting measures. Moreover, there is a natural $\mathcal{O}_K$-action on $G$ by translations: 
\begin{equation}\label{eq:rota}
T_a(g_1,g_2,\dots)=(g_1+a,g_2+a,\dots), a\in\OK.
\end{equation}
Since for each $L\ge 1$, the action of $(T_a)_{a\in\mathcal O_K}$ on the finite group $G_L:=\prod_{\ell=1}^L \mathcal O_K/b_\ell$ is transitive by the Chinese Remainder
Theorem for commutative rings (see e.g.\ Chapter I, \S 3 in \cite{MR1697859}), it is ergodic. This immediately implies that the system $(G,(T_a)_{a\in\mathcal O_K},\mathbb P)$ is ergodic. Moreover, since it is an ergodic rotation on a compact group, the action $(G,(T_a)_{a\in\mathcal O_K})$ is in fact uniquely ergodic.

Let $\varphi\colon G\to \{0,1\}^{\mathcal{O}_K}$ be defined as
\begin{equation}\label{eq:wzor}
\varphi(g)(a)=\begin{cases}
1,& \text{ if }g_{\ell}+a\not\equiv 0\bmod \mathfrak{b}_\ell \text{ for each }\ell\geq 1,\\
0,&\text{ otherwise},
\end{cases}
\end{equation}
where $g=(g_1,g_2,\dots)$. Notice that $\varphi(\underline{0})=\eta=\raz_{\mathcal{F}_\mathfrak{B}}$, where $\underline{0}=(0,0,\dots)$. 
\begin{Remark}
We have $\varphi={(\raz_C\circ T_a)}_{a\in\OK}$, where
\begin{equation}\label{eq:ce}
C=\{g\in G : g_\ell\not\equiv 0 \bmod \mathfrak{b}_\ell\text{ for each }\ell\geq 1\}.
\end{equation}
In other words, $\varphi$ is the coding of orbits of points under ${(T_a)}_{a\in\OK}$ with respect to the partition $\{C,G\setminus C\}$ of $G$.
\end{Remark}

Finally, let $\nu_\eta:=\varphi_\ast(\PP)$ be the pushforward of $\PP$ under $\varphi$.
We will call $\nu_\eta$ the \emph{Mirsky measure}. In the case of $\{p^k : p\in\mathcal{P}\}$-free numbers, in particular in the square-free case, this measure was considered by Mirsky~\cite{MR0021566,MR0028334} (cf.\ also~\cite{MR0030561}) who studied the frequencies of blocks, cf.\ \cref{thm:A}.

\paragraph{Admissible subshift}
Given a subset $A\subseteq \mathcal{O}_K$ and an ideal $\mathfrak{a}\subseteq \OK$, let
$$
D(\mathfrak{a}|A):=|A/\mathfrak{a}|=|\{b\bmod \mathfrak{a} : b\equiv a\bmod \mathfrak{a} \text{ for some }a\in A\}|.
$$
\begin{Def}[cf.\ \cite{MR3014544}]
We say that $A$ is {\em$\mathfrak{B}$-admissible} (or simply {\em admissible}) whenever
$$
D(\mathfrak{b}_\ell|A)<N(\mathfrak{b}_\ell)\text{ for each }\ell\geq 1.
$$
We say that $x\in \{0,1\}^{\mathcal{O}_K}$ is $\mathfrak{B}$-admissible if its support, denoted by $\text{supp }x$, is $\mathfrak{B}$-admissible; we will denote the set of all admissible sequences in $\{0,1\}^{\mathcal{O}_K}$  by $X_\mathfrak{B}$ (cf.\ Remark~\ref{uw:domk0}).
\end{Def}
\begin{Remark}\label{uw:domk0}
Notice that $X_\mathfrak{B}$ is a subshift. Indeed, it suffices to notice that if $x\in\{0,1\}^{\OK}$ is such that for each finite $B\subseteq \text{supp }x$,
$$
D(\mathfrak{b}_\ell|B)< N(\mathfrak{b}_\ell)\text{ for each }\ell\geq 1,
$$
then
$D(\mathfrak{b}_\ell|\text{supp }x)< N(\mathfrak{b}_\ell)\text{ for all }\ell\geq 1$.
\end{Remark}

\begin{Def}[cf.\ \cite{MR2317754,MR3007694}]
Let $Y\subseteq \{0,1\}^{\mathcal{O}_K}$ be a subshift. We say that $Y$ is \emph{hereditary} whenever $x,x'\in \{0,1\}^{\mathcal{O}_K}$ with $x\in Y$, $x'\leq x$ (coordinatewise) implies $x'\in Y$. 
\end{Def}
Clearly, $X_\mathfrak{B}$ is hereditary.

\section{Main results and their proofs}\label{se:results}
We are now ready to state our main results in their full form.
\begin{Thx}\label{thm:A}
For any Erd\H{o}s set $\mathfrak{B}$, we have the following:
\begin{enumerate}[(i)]
\item\label{thmAi} 
The Mirsky measure $\nu_\eta$ is invariant under ${(S_a)}_{a\in \mathcal{O}_K}$, and $\eta$ is generic for $\nu_\eta$ along a F\o{}lner sequence $(B_n)_{n\geq1}$.
\item\label{thmAii}
The dynamical systems $(X_\mathfrak{B},{(S_a)}_{a\in \mathcal{O}_K},\nu_\eta)$ and $(G,{(T_a)}_{a\in \mathcal{O}_K},\PP)$ are measure-theoretically isomorphic. In particular, $(X_\mathfrak{B},{(S_a)}_{a\in \mathcal{O}_K},\nu_\eta)$ is of zero Kolmogorov entropy.
\end{enumerate}
\end{Thx}
\begin{Thx}\label{thm:B}
For any Erd\H{o}s set $\mathfrak{B}=\{\mathfrak{b}_\ell : \ell\geq 1\}$, we have
\[h_{top}(X_\mathfrak{B},{(S_a)}_{a\in\OK})=\prod\limits_{\ell\geqslant1}\left(1-\frac{1}{N(\mathfrak{b}_\ell)}\right).\]
\end{Thx}
\begin{Thx}\label{thm:C}
For any Erd\H{o}s set $\mathfrak{B}$, we have $X_\eta=X_\mathfrak{B}$.
\end{Thx}
\setcounter{Thx}{4}
\begin{Thx}\footnote{We do not have Theorem D, to keep the names of our main results consistent with (A)-(E) used in the Introduction}\label{thm:E}
For any Erd\H{o}s set $\mathfrak{B}$, $(X_\mathfrak{B},{(S_a)}_{a\in\OK})$ has a non-trivial topological joining with $(G,{(T_a)}_{a\in\OK})$.
\end{Thx}

\begin{Remark}
Since $(G,{(T_a)}_{a\in\OK})$ is minimal and distal,\footnote{Recall that ${(T_a)}_{a\in \OK}$ is said to be \emph{distal} whenever $\inf_{a\in\OK} d(T_a x, T_a y)>0$ for all $x\neq y$.} it follows by \cref{thm:E} and by Theorem II.3 in~\cite{MR0213508} that $(X_\mathfrak{B},{(S_a)}_{a\in\OK})$ fails to be topologically weakly mixing: its Cartesian square is not transitive. On the other hand, the proximality of $(X_\mathfrak{B},{(S_a)}_{a\in\OK})$ implies that its maximal equicontinuous factor is trivial.
\end{Remark}

\subsection{Proof of \cref{thm:A} \eqref{thmAi}}

\begin{Prop}\label{lm31}
   \label{lm:Counting under restriction}
	Let $K$ be a number field of degree $d$ and $\gb\neq \{0\}$ an ideal of $\co_K$. For any $a \in \co_K$, 
    \[
    |B_n\cap (\mathfrak{b}+a)| = \frac{|B_n|}{N(\gb)} + \textrm{O}\left(1+\max_{1\leq j \leq d-1}\frac{n^j}{\lambda_1(\gb)\dots\lambda_j(\gb)}\right),
    \] 
    where the constant on the error term depends on $K$ only.
\end{Prop}
For the proof of the above lemma, we will need a result from~\cite{Widmer}. Before we formulate it, we need to introduce some notation. Let $\text{Lip}(d,c,M,L)$ be the family of all sets $S\subset \R^d$ such that there exist maps $\phi_1,\dots, \phi_M\colon [0,1]^{d-c}\to \R^d$ such that
\begin{itemize}
\item 
for each $1\leq j\leq M$, $\phi_j$ is Lipschitz with constant $L$ with respect to the corresponding Euclidean norms,
\item 
$S\subset \bigcup_{1\leq j\leq M}\phi_j([0,1]^{d-c})$.
\end{itemize}
Additionally, given a lattice $\Lambda\subset \R^d$, we denote by $\det(\Lambda)$ the Lebesgue measure of any fundamental domain of $\Lambda$.
\begin{Th}[Theorem 5.4 in~\cite{Widmer}]\label{th54}
If $\Lambda\subset \R^d$ is a lattice and $S\subset\R^d$ is bounded with $\partial S\in \text{Lip}(d,1,M,L)$ then
\[
\left|S\cap\Lambda\right| = \frac{\vol(S)}{\det(\Lambda)} + C\cdot M\left(1+\max_{1\leq j < d}\frac{L^j}{\lambda_1(\Lambda)\cdots\lambda_j(\Lambda)}\right),
\]
where $C$ is a constant depending only on $d$.
\end{Th}
\begin{proof}[Proof of Proposition~\ref{lm31}]
Consider $S = -t+ [-n,n]^d$ (where $t$ will be chosen later). Clearly, $\partial S \in \text{Lip}(d,1,2d,2n)$, since each of the $2d$ faces that contribute to the boundary of $\partial S $ can be parameterized by a map with Lipschitz constant $2n$ (for example, take the map $\phi\colon [0,1]^{d-1} \rightarrow [-n,n]^{d-1}\times\{n\}$ given by $\phi(x_1,\dots,x_{d-1}) = -t+ (2nx_1-n,\dots,2nx_{d-1}-n,n)$).
Moreover, $\vol(S) = (2n)^d$. It follows by Theorem~\ref{th54} that
\begin{equation}\label{wzorek}
|[-n,n]^d\cap(t+\Lambda)|=|S\cap \Lambda|=\frac{2^dn^d}{\det(\Lambda)} + \textrm{O}_d\left(1+\max_{1\leq j < d}\frac{n^j}{\lambda_1(\Lambda)\cdots\lambda_j(\Lambda)}\right)
\end{equation}
for any lattice $\Lambda\subset \R^d$.

Notice that
\[
|B_n\cap (\mathfrak{b}+a)|=|\sigma^{-1}([-n,n]^d)\cap (\sigma^{-1}(\sigma(\mathfrak{b}))+\sigma^{-1}(\sigma(a)))|=|[-n,n]^d\cap (\sigma(\mathfrak{b})+\sigma(a))|.
\]
Therefore, formula~\eqref{wzorek} for $t=\sigma(a)$ and $\Lambda=\sigma(\mathfrak{b})$ (recall that $\sigma$ stands for the Minkowski embedding) yields the following:
\begin{align}
\begin{split}\label{row1}
|B_n\cap (\mathfrak{b}+a)|&=\frac{2^dn^d}{\det(\sigma(\mathfrak{b}))}+\textrm{O}_d\left(1+\max_{1\leq j < d}\frac{n^j}{\lambda_1(\sigma(\mathfrak{b}))\cdot\ldots\cdot \lambda_j(\sigma(\mathfrak{b}))} \right)\\
&=\frac{2^dn^d}{\det(\sigma(\mathfrak{b}))}+\textrm{O}_d\left(1+\max_{1\leq j < d}\frac{n^j}{\lambda_1(\mathfrak{b})\cdot\ldots\cdot \lambda_j(\mathfrak{b})}\right)
\end{split}
\end{align}
In particular, for $a=0$ (hence $t=0$) and $\mathfrak{b}=\co_K$, we have
\begin{equation}\label{row2}
|B_n|=\frac{2^dn^d}{\det(\sigma(\co_K))}+\textrm{O}(n^{d-1}).
\end{equation}
By Proposition 5.2. in \cite{MR1697859} and the preceding discussion, for any non-zero ideal $\mathfrak{a}$ of $\co_K$, there is a constant $c_K$ only depending on $K$ such that $\det(\sigma(\mathfrak{a})) = c_KN(\mathfrak{a})$. In particular $\det(\sigma(\co_K)) = c_K$. Using~\eqref{row2} it follows that $$\frac{2^dn^d}{\det(\sigma(\mathfrak{b}))} = \frac{2^dn^d}{c_KN(\gb)} = (|B_n|+\textrm{O}(n^{d-1}))\cdot \frac{1}{N(\mathfrak{b})}.$$

Combining this with~\eqref{row1}, we conclude that
\begin{align*}
|B_n\cap (\mathfrak{b}+a)|&=(|B_n|+\textrm{O}(n^{d-1}))\cdot \frac{1}{N(\mathfrak{b})}+\textrm{O}_d\left(1+\max_{1\leq j < d}\frac{n^j}{\lambda_1(\mathfrak{b})\cdot\ldots\cdot \lambda_j(\mathfrak{b})}\right)\\
&=\frac{|B_n|}{N(\mathfrak{b})}+\textrm{O}_d\left(\frac{n^{d-1}}{N(\mathfrak{b})}\right)+\textrm{O}_d\left(1+\max_{1\leq j < d}\frac{n^j}{\lambda_1(\mathfrak{b})\cdot\ldots\cdot \lambda_j(\mathfrak{b})}\right)\\
&=\frac{|B_n|}{N(\mathfrak{b})}+\textrm{O}\left(1+\max_{1\leq j < d}\frac{n^j}{\lambda_1(\mathfrak{b})\cdot\ldots\cdot \lambda_j(\mathfrak{b})}\right)
\end{align*}
where the last equality follows by~\eqref{eq: product of Lambdas} after noticing that $\lambda_d(\mathfrak{b}) \geq 1$, since $\|x\|\geq 1$ for any $x\in \co_K$ (we point out that the constant in~\eqref{eq: product of Lambdas} depends on $K$, therefore the obtained bound also depends on $K$, not only on $d$). 
\end{proof}
\begin{Prop}\label{strong_light_tails}
    Suppose that $\mathfrak{B} = \{\gb_\ell : \ell\geq 1\}$ is Erd\H{o}s. Then $\lim_{L \rightarrow \infty} d\left(\bigcup_{\ell>L} \gb_\ell\right) = 0.$
\end{Prop}
\begin{proof}
 Let $x \in B_n \setminus \{0\}$. If $x \in \gb_\ell$, then we must have that $ \lambda_1(\gb_\ell) \leq \|x\| \leq n$. Therefore, we have that $$\left|B_n \cap \bigcup_{\ell>L}\gb_\ell\right| \leq  1 + \sum_{\substack{\ell: \lambda_1(\gb_\ell) \leq n  \\ \ell> L}} |\{x \in  B_n \setminus \{0\} : x \in \gb_\ell \}|  , $$ which, after applying  \Cref{lm:Counting under restriction} gives 
	\[\left|B_n \cap \bigcup_{\ell>L}\gb_\ell\right| \leq 1+  \sum_{\substack{\ell: \lambda_1(\gb_\ell) \leq n  \\ \ell> L}}\left( \frac{|B_n|}{N(\gb_\ell)} + O\left(1+\max_{1\leq j < d}\frac{n^{j}}{\lambda_1(\gb_\ell)\dots \lambda_j(\gb_\ell)}\right)  \right).\]
	
	We have to deal with three distinct sums separably, and show that once we divide by $|B_n|$, and take the limit of $n$ and then $L$ to infinity, these will go to $0$.  First, notice that  $$\lim_{L \rightarrow \infty} \lim_{n\rightarrow\infty} \frac{1}{|B_n|} \sum_{\substack{\ell: \lambda_1(\gb_\ell) \leq n  \\ \ell> L}} \frac{|B_n|}{N(\gb_\ell)} \leq \lim_{L \rightarrow \infty} \sum_{\ell>L} \frac{1}{N(\gb_\ell)} = 0, $$ as the series converges by hypothesis, so the first sum is dealt with.
	
	We next have to show that $$ \lim_{L \rightarrow \infty} \lim_{n\rightarrow\infty} \frac{1}{|B_n|} \sum_{\substack{\ell: \lambda_1(\gb_\ell) \leq n  \\ \ell> L}}1 = 0.$$ By  \Cref{lm: lambda_i asymp n(gb)(1/n)}, there is some $C$ dependent only on $K$ such that if $N(\gb_\ell)\leq C n^d$, then $\lambda_1(\gb_\ell) \leq n$. Therefore, the sum is bounded up to a constant multiple by $$\frac{1}{|B_n|}\sum_{\ell:N(\gb_\ell)\leq C n^d}1 .$$  Since all the $\gb_\ell$ are coprime, the number of ideals in $\mathfrak{B}$ with norm smaller than $Cn^d$ must be bounded by the number of prime ideals with norm smaller than $Cn^d$. By  \Cref{thm: Prime Ideal Theorem}, this number is bounded by $cn^d/\log(n)$ for some constant $c$ depending only on $K$. Consequently, it follows that  $$\lim_{L \rightarrow \infty} \lim_{n\rightarrow\infty}  \frac{1}{|B_n|}\sum_{\substack{\ell: \lambda_1(\gb_\ell) \leq n  \\ \ell> L}} 1 \ll \lim_{L \rightarrow \infty} \lim_{n\rightarrow\infty} \frac{n^d}{|B_n|\log(n)} = 0, $$ as we wanted to show.
	
	We are left with showing that for any $1 \leq j \leq d-1 $, we have $$\lim_{L \rightarrow \infty} \lim_{n\rightarrow\infty}  \frac{1}{|B_n|}\sum_{\substack{\ell: \lambda_1(\gb_\ell) \leq n  \\ \ell> L}} \frac{n^j}{\lambda_1(\gb_\ell)\dots \lambda_j(\gb_\ell)} = 0.$$
	Fix $j$. Using Equation (\ref{eq: product of Lambdas}) we have that $$\frac{n^j}{\lambda_1(\gb_\ell)\dots \lambda_j(\gb_\ell)} \asymp_K \frac{\lambda_{j+1}(\gb_\ell)\dots\lambda_d(\gb_\ell)n^{j}}{N(\gb_\ell)}. $$ By  \Cref{lm: lambda_i asymp n(gb)(1/n)}, we know that $\lambda_1(\gb_\ell) \asymp_K \lambda_d(\gb_\ell)  $, so there is some $C$ depending only on $K$, such that $\lambda_d(\gb_\ell) \leq C\lambda_1(\gb_\ell) $. Therefore, for any $i$ such that $\lambda_1(\gb_\ell) \leq n$, we have that $\lambda_d(\gb_\ell) \leq Cn$. Hence, $$\frac{1}{|B_n|}\sum_{\substack{\ell: \lambda_1(\gb_\ell) \leq n  \\ \ell> L}} \frac{\lambda_{j+1}(\gb_\ell)\dots\lambda_d(\gb_\ell)n^{j}}{N(\gb_\ell)} \leq \frac{1}{|B_n|}\sum_{\substack{\ell: \lambda_d(\gb_\ell) \leq Cn  \\ \ell> L}} \frac{\lambda_d(\gb_\ell)^{d-j}n^{j}}{N(\gb_\ell)} \leq  \frac{n^{n}}{|B_n|}\sum_{\substack{\ell: \lambda_d(\gb_\ell) \leq Cn  \\ \ell> L}} \frac{C^{d-j}}{N(\gb_\ell)}.   $$ The term $n^d/|B_n| $ is bounded by a constant only depending on $K$, so it follows that $$\lim_{L \rightarrow \infty} \lim_{n\rightarrow\infty}  \frac{1}{|B_n|}\sum_{\substack{\ell: \lambda_1(\gb_\ell) \leq n  \\ \ell> L}} \frac{\lambda_{j+1}(\gb_\ell)\dots\lambda_d(\gb_\ell)n^{j}}{N(\gb_\ell)} \ll \lim_{L \rightarrow \infty} \sum_{\substack{ \ell> L}} \frac{1}{N(\gb_\ell)}  = 0.$$

	Since all these limits go to $0$, we conclude that $$\lim_{L \rightarrow \infty} \lim_{n\rightarrow\infty}  \frac{1}{|B_n|}\left|B_n \cap \bigcup_{\ell>L}\gb_\ell\right| =0,$$ as we wanted to show. 
\end{proof}

Now, notice that
\begin{equation}\label{lm:varphi-eq}
S_a\circ \varphi=\varphi\circ T_a\text{ for each }a\in\mathcal{O}_K.
\end{equation}
Indeed, we have
\begin{align*}
\varphi\circ T_a(g)(b)=1 &\iff {(T_a(g))}_\ell+b\not\equiv 0\bmod \mathfrak{b}_\ell \text{ for each }\ell\geq 1\\
&\iff g_{\ell}+a+b\not\equiv 0\bmod \mathfrak{b}_\ell \text{ for each }\ell\geq 1\\
&\iff \varphi(g)(b+a)=1\\
&\iff S_a\circ \varphi(g)(b)=1.
\end{align*}
In particular, the Mirsky measure $\nu_\eta$ is invariant under ${(S_a)}_{a\in \mathcal{O}_K}$.

We will now prove that $\eta$ is generic for $\nu_\eta$ along $(B_n)_{n\geq1}$. The main idea here comes from the proof of Theorem 4.1.\ in \cite{MR3428961}.
In view of Remark~\ref{uw:geny}, we only need to show that
\begin{equation}\label{eq:claim}
\frac{1}{|B_n|}\sum_{a\in B_n}\raz_{C_B^0}(S_a \eta)=\frac{1}{|B_n|}\sum_{a\in B_n}\raz_{\varphi^{-1}(C_B^0)}(T_a \underline{0}) \to\nu_\eta(C_B^0) =\PP(\varphi^{-1}(C_B^0))
\end{equation}
for each finite set $B\subseteq \mathcal{O}_K$ (in the left equality we use the definition of $\eta$ and \eqref{lm:varphi-eq}). We have
\begin{equation}\label{eq:z}
\varphi^{-1}(C_B^0)=\bigcap_{b\in B}T_{-b} (\varphi^{-1}(C_0^0))=\bigcap_{b\in B}T_{-b} C^c,
\end{equation}
where $C$ is as in~\eqref{eq:ce}, i.e.\ $C=\varphi^{-1}(C_0^1)$.
Moreover, for each $L\geq 1$,
\begin{equation}\label{eq:a}
\bigcap_{b\in B} T_{-b}C_L^c\subseteq \bigcap_{b\in B}T_{-b} C^c \subseteq \bigcap_{b\in B} T_{-b}C_L^c \cup \bigcup_{b\in B}T_{-b}(C^c\setminus C_L^c),
\end{equation}
where $C_L:=\{g\in G : g_\ell\not\equiv 0\bmod \mathfrak{b}_\ell \text{ for each }1\leq \ell\leq L\}$. Since each $C_L$ is clopen, it follows that the function $\raz_{\bigcap_{b\in B}T_{-b}C_L^c}$ is continuous. Thus, since $(G,{(T_a)}_{a\in\OK})$ is uniquely ergodic, by Remark~\ref{uw:uniqergo}, we obtain
\begin{equation}\label{eq:b}
\frac{1}{|B_n|}\sum_{a\in B_n}\raz_{\bigcap_{b\in B}T_{-b}C_L^c}(T_a \underline{0}) \to \PP(\bigcap_{b\in B}T_{-b}C_L^c) \text{ as }n\to\infty.
\end{equation}
Moreover, given $\vep>0$, for $L$ sufficiently large, 
\begin{equation}\label{eq:c}
\PP(\bigcap_{b\in B}T_{-b}C_L^c)\geq \PP(\bigcap_{b\in B}T_{-b}C^c)-\vep
\end{equation}
and
\begin{align}
\begin{split}\label{eq:d}
\limsup_{n\to \infty}&\frac{1}{|B_n|}\sum_{a\in B_n}\raz_{\bigcup_{b\in B}T_{-b}(C^c\setminus C_L^c)}(T_a \underline{0})\leq |B|\limsup_{n\to\infty}\frac{1}{|B_n|}\sum_{a\in B_n}\raz_{(C^c\setminus C_L^c)}(T_a \underline{0})\\
&=|B|\limsup_{n\to\infty}\frac{|\left(\bigcup_{\ell\geq1}\mathfrak{b}_\ell\setminus\bigcup_{\ell\leq L}\mathfrak{b}_{\ell}\right)\cap B_n|}{|B_n|}\leq|B|\limsup_{n\to\infty}\frac{|\bigcup_{\ell>L}\mathfrak{b}_\ell\cap B_n|}{|B_n|}\\
&=|B|\cdot\overline{d}\left(\bigcup_{\ell>L}\mathfrak{b}_\ell\right).
\end{split}
\end{align}
By Proposition~\ref{strong_light_tails}, the right hand side of \eqref{eq:d} tends to $0$ as $L\to\infty$.
Using~\eqref{eq:z}, \eqref{eq:a}, \eqref{eq:b}, \eqref{eq:c} and~\eqref{eq:d}, we conclude that \eqref{eq:claim} indeed holds, and the proof of \cref{thm:A} \eqref{thmAi} is complete.

\subsection{Proof of \cref{thm:C}}
We begin this section by the following simple observation which yields one of the inclusions in the assertion of \cref{thm:C}:
\begin{Lemma}\label{eq:orbcls}
For any Erd\H{o}s set $\mathfrak{B}$, $\varphi(G)\subseteq X_\mathfrak{B}$. In particular, $X_\eta\subseteq X_\mathfrak{B}$ and $\nu_\eta(X_\mathfrak{B})=1$.
\end{Lemma}
\begin{proof}
Let $a\in\text{supp }\varphi(g)$, i.e.\ $g_{\ell}+a\not\equiv 0\bmod \mathfrak{b}_\ell$  for each $\ell\geq 1$. In other words, $a\not\equiv -g_\ell\bmod \mathfrak{b}_\ell$, which yields $-g_\ell\bmod \mathfrak{b}_\ell\not\in\text{supp } \varphi(g)/\mathfrak{b}_\ell$ for each $\ell\geq 1$.
\end{proof}
The proof of the other inclusion $X_\mathfrak{B}\subseteq X_\eta$ is a bit more involved. It is an immediate consequence of \cref{thm:A}~\eqref{thmAi} and the following result:
\begin{Prop}[cf.\ Proposition 2.5.\ in \cite{MR3428961}]\label{pr:dodatnie}
Let $A,B\subseteq \mathcal{O}_K$ be finite and disjoint. For any Erd\H{o}s set $\mathfrak{B}$, the following are equivalent:
\begin{enumerate}[(i)]
	\item $A$ is $\mathfrak{B}$-admissible,\label{i}
	\item $\nu_\eta(C_A^1)>0$,\label{ii}
	\item $\nu_\eta(C_{A,B})>0$.\label{iii}
\end{enumerate}
\end{Prop}
Before giving the proof, let us point out that we obtain  the following corollary as another immediate consequence of \cref{thm:A}~\eqref{thmAi}  and Proposition~\ref{pr:dodatnie}:
\begin{Cor}\label{co:rowne}
For any Erd\H{o}s set $\mathfrak{B}$, the topological support of $\nu_{\eta}$ is the subshift $X_\mathfrak{B}$ of $\mathfrak{B}$-admissible sequences.
\end{Cor}

For the proof of Proposition~\ref{pr:dodatnie}, we will need two lemmas.
\begin{Lemma}\label{lm:miary}
Suppose that $\mathfrak{B}$ is Erd\H{o}s. Then
for any finite set $A\subseteq \mathcal{O}_K$, we have $\nu_\eta(C_A^1)=\prod_{\ell\geq 1}\left(1-\frac{D(\mathfrak{b}_\ell|A)}{N(\mathfrak{b}_\ell)} \right)$.
\end{Lemma}
\begin{proof}
For each finite $A\subseteq \mathcal{O}_K$, we have
\begin{align*}
\nu_\eta(C_A^1)&=\varphi_\ast(\PP)(C_A^1)=\PP(\varphi^{-1}(C_A^1))\\
&=\PP\left(\bigcap_{\ell\geq 1}\{g\in G : g_\ell+a\not\equiv 0\bmod \mathfrak{b}_\ell\text{ for }a\in A\}\right)=\prod_{\ell\geq 1}\left(1-\frac{D(\mathfrak{b}_\ell|A)}{N(\mathfrak{b}_\ell)} \right).
\end{align*}
This finishes the proof.
\end{proof}
\begin{Remark}
It follows from Lemma~\ref{lm:miary} and Lemma 2.3.\ in~\cite{MR3428961} that
\[
\nu_\eta(C_{A,B})=\sum_{A\subseteq D\subseteq A\cup B}(-1)^{|D\setminus A|}\prod_{\ell\geq 1}\left(1-\frac{D(\mathfrak{b}_\ell|A)}{N(\mathfrak{b}_\ell)} \right)
\]
for each pair $A,B\subseteq \mathcal{O}_K$ of finite disjoint sets and any Erd\H{o}s set $\mathfrak{B}$.
\end{Remark}
\begin{Remark}\label{lm:rozw} 
Since any Dedekind domain is a Noetherian ring, by the Gilmer and Heinzer Theorem \cite{Gil-He}, there are only finitely many ideals of any fixed index in $\OK$. So any non-zero element of $\OK$ is contained in finitely many ideals. Hence the intersection of an infinite collection of pairwise coprime (proper) ideals is trivial.
\end{Remark}
\begin{proof}[Proof of~Proposition~\ref{pr:dodatnie}]
By \cref{thm:A}~\eqref{thmAi},~\eqref{iii} implies~\eqref{i}. Fix a finite admissible set $A\subseteq \mathcal{O}_K$. By~\cref{lm:miary}, we obtain
$$
\nu_{\eta}(C_A^1)=\prod_{\ell\geq 1}\left(1-\frac{D(\mathfrak{b}_\ell | A)}{N(\mathfrak{b}_\ell)} \right)>0
 \iff \sum_{\ell\geq 1}\frac{D(\mathfrak{b}_\ell|A)}{N(\mathfrak{b}_\ell)}<\infty
 \iff \sum_{\ell\geq 1}\frac{1}{N(\mathfrak{b}_\ell)}<\infty,
$$
whence~\eqref{i} implies~\eqref{ii}.

It remains to show that~\eqref{ii} implies~\eqref{iii}. Fix finite disjoint sets $A,B$. It follows by Remark~\ref{lm:rozw} that there exists $L\geq 1$ such that $a\equiv b\bmod \mathfrak{b}_\ell$ has no solution in $a\in A$, $b\in B$ for $\ell>L$. Let $B=\{b_1,\dots, b_r\}$ and consider
\begin{multline*}
\{g\in G : \forall 1\leq j\leq r,\ g_{L+j}+b_j\equiv 0\bmod \mathfrak{b}_{L+j}\}\\
\cap \{g\in G : \forall \ell\not\in \{L+1,\dots,L+r\}\ \forall a\in A,\ g_{\ell}+a\not\equiv 0 \bmod \mathfrak{b}_\ell \}\subseteq \varphi^{-1}(C_{A,B})
\end{multline*}
(the inclusion follows by the choice of $L$). The left-hand side of the above formula is an intersection of two independent events in $(G,\PP)$. The first of them has probability $\prod_{j=1}^r \frac{1}{N(\mathfrak{b}_{L+j})}>0$, and the second contains $\varphi^{-1}(C_A^1)$, therefore has also positive probability.
\end{proof}
\begin{Remark}\label{her}
An immediate consequence of \cref{thm:C} is that $X_\eta$ is hereditary for any Erd\H{o}s set $\mathfrak{B}$.
\end{Remark}

\subsection{Proof of \cref{thm:B} (and beyond)}
The main purpose of this section is to prove \cref{thm:B}. However, we will not only compute the topological entropy of $(X_\mathfrak{B},{(S_a)}_{a\in\OK})$, but also of its restriction to some natural invariant subsets of $X_\mathfrak{B}$. This will be crucial later, in the proof of \cref{thm:A}~\eqref{thmAii}.

For $s_\ell\geq 1$, $\ell\geq1$, let $\underline{s}:=(s_\ell)_{\ell\geq1}$. Consider
\begin{align}
\begin{split}\label{eq:pozdr}
Y_{\underline{s}}:=&\{x\in {X}_\mathfrak{B} : D(\mathfrak{b}_\ell|\text{supp }x)=N(\mathfrak{b}_\ell)-s_\ell \text{ for } \ell\geq 1 \},\\
Y_{\geq \underline{s}}:=&\{x\in {X}_\mathfrak{B} : D(\mathfrak{b}_\ell|\text{supp }x)\leqslant N(\mathfrak{b}_\ell)-s_\ell \text{ for }\ell\geq 1\}.
\end{split}
\end{align}
For $\underline{s}=(1,1,\dots)$ we will simply write $Y$ instead of $Y_{\underline{s}}$. Notice that we have
\begin{equation}\label{rozkla}
X_\mathfrak{B}=\bigcup_{s_\ell\geq 1, \ell\geq 1} Y_{(s_\ell)_{\ell\geq1}}.
\end{equation}
\begin{Remark}[cf.\ Remark~\ref{uw:domk0}] \label{uw:domk}
Notice that for any $\mathfrak{B}$, each $Y_{\geq \underline{s}}\subseteq X_\mathfrak{B}$ is closed and invariant under ${(S_a)}_{a\in\OK}$. Moreover, $\overline{Y}_{\underline{s}}\subseteq Y_{\geq \underline{s}}$.
\end{Remark}

Fix a F\o{}lner sequence ${(F_n)}_{n\geq 1}\subseteq \OK$. For each choice of $\underline{s}=(s_\ell)_{\ell\geq1}$, let
$$
\mathcal{F}_{n}^{\geq \underline{s}}:=\{W\subseteq F_n : D(\mathfrak{b}_\ell|W)\leqslant N(\mathfrak{b}_\ell)-s_\ell\text{ for }\ell\geq 1\}
$$
and let $\gamma^{\geq \underline{s}}(n)$ denote the cardinality of $\mathcal{F}_{n}^{\geq \underline{s}}$. In particular, $\gamma^{\geq \underline{1}}(n)$, where $\underline{1}=(1,1,\dots)$, denotes the number of $\mathfrak{B}$-admissible subsets of $F_n$. Moreover, given $L\geq 1$, let $\underline{s}_L:=(s_1,\ldots,s_L)$ and
$$
\mathcal{F}_{n,L}^{\geq \underline{s}_L}:=\{W\subseteq F_n : D(\mathfrak{b}_\ell|W)\leqslant N(\mathfrak{b}_\ell)-s_\ell\text{ for }1\leq \ell\leq L\}
$$
and let $\gamma_L^{\geq \underline{s}_L}(n)$ be the cardinality of $\mathcal{F}_{n,L}^{\geq \underline{s}_L}$. In particular, $\gamma_L^{\geq \underline{1}_L}(n)$, where $\underline{1}_L=(\underbrace{1,1\dotsc,1}_L)$, denotes the number of $\mathfrak{B}_L$-admissible subsets of $F_n$, where $\mathfrak{B}_L=\{\mathfrak{b}_\ell : 1\leq \ell\leq L\}$ and $\mathfrak{B}_L$-admissibility is defined in a similar way as $\mathfrak{B}$-admissibility. Clearly, 
$$
\gamma^{\geq \underline{s}}(n)\leqslant\gamma^{\geq \underline{s}_L}_L(n) \text{ for each }n\geq 1, L\geq 1.
$$
Moreover, given $n\geq 1$, $\gamma^{\geq \underline{s}_L}_L(n)$ decreases to $\gamma^{\geq \underline{s}}(n)$, and 
\begin{equation}\label{eq:finit}
\gamma^{\geq \underline{s}}(n)=\gamma^{\geq \underline{s}_{L(n)}}_{L(n)}(n)
\end{equation}
for some $L(n)\geq 1$.

Finally, for each choice of $\emptyset \neq A_\ell \subseteq \OK/{\mathfrak{b}_\ell}$, let
$$
Z_L=Z_L(A_1,\dots, A_L)=\{x\in\OK : x\bmod \mathfrak{b}_\ell \not\in A_\ell : 1\leq \ell\leq L\}.
$$
Notice that, for each $n\geq 1$, $F_n\setminus Z_L\in\mathcal{F}_{n,L}^{\geq \underline{s}_L}$ such that $s_\ell=|A_\ell|$, $1\leq \ell\leq L$. In particular, $F_n\setminus Z_L$ is $\mathfrak{B}_L$-admissible.

\begin{Lemma}\label{l:frek}
Fix $\mathfrak{B}$. For arbitrary $\vep>0$ and $n\in\N$ sufficiently large
\begin{equation}\label{frek}
\prod_{\ell=1}^{L}\left(1-\frac{s_\ell}{N(b_\ell)} \right)-\vep<\frac{|F_n\setminus Z_L|}{|F_n|}<\prod_{\ell=1}^{L}\left(1-\frac{s_\ell}{N(b_\ell)} \right)+\vep.
\end{equation}
\end{Lemma}
\begin{proof}
Recall that $G_L=\prod\limits_{\ell=1}^L \mathcal{O}_K/{\mathfrak{b}_\ell}$ and
put 
$$
D_L:=\{g\in G_L  :  g_\ell\not\in A_\ell \text{ for }1\leqslant \ell\leqslant L\}.
$$
Since $\raz_{D_L}$ is continuous ($D_L$ is clopen), it follows by the unique ergodicity of the restriction of ${(T_a)}_{a\in\OK}$ to the first $L$ coordinates of $G$, i.e.\ to $G_L$, that
\[
\frac{1}{|F_n|} \sum_{a\in F_n} \raz_{D_L} (T_a (\underbrace{0,\dotsc,0}_L)) \rightarrow \PP(D_L)=\prod\limits_{\ell=1}^{L}\left(1-\frac{s_\ell}{N(\mathfrak{b}_\ell)} \right)
\]
(cf.\ Remark~\ref{uw:uniqergo}). Moreover,
$$
\raz_{D_L}(T_a(0,\dotsc,0))=1 \iff a\bmod \mathfrak{b}_\ell \not\in A_\ell \text{ for }1\leqslant l\leqslant L \iff a\not\in Z_L,
$$
whence
$$
\sum_{a\in F_n} \raz_{D_L} (T_a (0,\dotsc,0))=|F_n\setminus Z_L|,
$$
which completes the proof.
\end{proof}
\begin{Lemma}\label{l:ensz}
For arbitrary $\vep>0$ and $n\in\N$ sufficiently large
$$
2^{|F_n|\left(\prod_{\ell=1}^{L} \ \left(1-\frac{s_\ell}{N(\mathfrak{b}_\ell)}\right)- \varepsilon\right)}
\leqslant\gamma^{\geq \underline{s}_L}_L(n)
\leqslant \prod_{\ell=1}^{L}{N(\mathfrak{b}_\ell)\choose s_\ell}\cdot 2^{|F_n|\left(\prod_{\ell=1}^{L} \ \left(1-\frac{s_\ell}{N(\mathfrak{b}_\ell)}\right)+ \varepsilon\right)}.
$$
\end{Lemma}
\begin{proof}
Fix $\vep>0$. Let $n\in\N$ be sufficiently large, so that~\eqref{frek} holds. The following procedure yields all elements of $\mathcal{F}_{n,L}^{\geq \underline{s}_L}$:
\begin{enumerate}[(a)]
\item choose $A_\ell \subseteq \OK/\mathfrak{b}_\ell$ with $|A_\ell|=s_\ell$, $1\leq \ell\leq L$,\label{step:A}
\item choose $W\subseteq F_n \setminus Z_L$, where $Z_L=Z_L(A_1,\dots, A_L)$\label{step:b}
\end{enumerate}
(some elements of $\mathcal{F}_{n,L}^{\geq \underline{s}_L}$ can be obtained in more than one way).
It follows from \cref{l:frek} that once we have fixed $A_1,\dots, A_L$ in step (\ref{step:A}), then the number of distinct elements of $\mathcal{F}_{n,L}^{\geq \underline{s}_L}$ obtained in step~(\ref{step:b}) can be estimated from below and from above by
$$
2^{|F_n|\left(\prod_{\ell=1}^{L} \ \left(1-\frac{s_\ell}{N(\mathfrak{b}_\ell)}\right)- \varepsilon\right)}\text{ and }2^{|F_n|\left(\prod_{\ell=1}^{L} \ \left(1-\frac{s_\ell}{N(\mathfrak{b}_\ell)}\right)+ \varepsilon\right)},
$$
respectively. Moreover, there are $\prod_{\ell=1}^{L}{N(\mathfrak{b}_\ell) \choose s_\ell}$ possible choices in step~(\ref{step:A}), which completes the proof.
\end{proof}

For the further discussion, we will use a particular F\o{}lner sequence. Let  
\begin{equation}\label{fo:17}
\text{$\iota \colon \Z^d \to \OK$ be a group isomorphism}
\end{equation}
 (recall that $\OK$ is isomorphic to a lattice in $\R^d$ via the Minkowski embedding, and any two lattices in $\R^d$ are isomorphic). Let ${(H_n)}_{n\geq 1}\subseteq \OK$ be the F\o{}lner sequence defined in the following way:
\begin{equation}\label{naszfolner}
H_n:=\{x\in\OK  :  \forall_{1\leqslant s\leqslant d} \ |\pi_s(\iota^{-1}(x))|\leqslant n\},
\end{equation}
where $\pi_s\colon \Z^d\to \Z$ is the projection onto the $s$-th coordinate.
\begin{Lemma}\label{l:pom}
For the F\o{}lner sequence ${(H_n)}_{n\geq 1}$ defined in~\eqref{naszfolner}, we have
\begin{equation}\label{eq:pom}
\gamma_L^{\geq \underline{s}_L}(nm)\leqslant\gamma_L^{\geq \underline{s}_L}(n)^{m^d}
\end{equation}
for any $n,m,L\geq 1$.
\end{Lemma}
\begin{proof}
We have $H_{nm}=\bigcap_{t=1}^d \bigcup_{j_t=1}^{m} H_{nm}^{(j_t)}=\bigcup_{j_1,\dotsc,j_d=1}^m \bigcap_{t=1}^d H_{nm}^{(j_t)}$,
where
$$
H_{nm}^{(j_t)}=\{x\in\OK : n(2j_t-m-1)-n\leqslant\pi_t(\iota^{-1}(x))\leqslant n(2j_t-m-1)+n\}.
$$
For $j_1,\dots,j_d\in \{1,\dots,m\}$, let $u_t:=2j_t-m-1$, $1\leq t\leq d$. Then, since $\iota$ is an isomorphism, we have
\begin{align*}
x\in\bigcap_{t=1}^d H_{nm}^{(j_t)} & \iff n u_t - n \leq \pi_t(\iota^{-1}(x))\leq nu_t+n\text{ for } 1\leq t\leq d\\
& \iff -n \leq \pi_t (\iota^{-1}(x-\iota(nu_1,\dots,nu_d))) \leq n\text{ for } 1\leq t\leq d\\
& \iff x-\iota(nu_1,\dots,nu_d)\in H_n.
\end{align*}
Thus,
$$
\bigcap_{t=1}^{d}H_{nm}^{(j_t)}=H_n+\iota(nu_1,\ldots,nu_d).
$$
Since the number of subsets $W\subseteq H_n+\iota(nu_1,\ldots,nu_d)$ satisfying $D(\mathfrak{b}_\ell|W)\leqslant N(\mathfrak{b}_\ell)-s_\ell, 1\leqslant\ell\leqslant L$ is equal to $\gamma_L^{\geq \underline{s}_L}(n)$, we conclude that~\eqref{eq:pom} indeed holds.
\end{proof}
\begin{Th}\label{pr:entropy}
For any Erd\H{o}s set $\mathfrak{B}$, we have \[h_{top}(Y_{\geq \underline{s}},{(S_a)}_{a\in\OK})=\prod\limits_{\ell\geq1}\left(1-\frac{s_\ell}{N(\mathfrak{b}_\ell)}\right).
\]
\end{Th}
\begin{proof}
We will use the F\o{}lner sequence ${(H_n)}_{n\geq 1}$ for calculation. We need to prove that 
\begin{equation}\label{eq:teza}
\lim_{n\to\infty} \ \frac{1}{|H_n|} \log_2 \gamma^{\geq \underline{s}}(n)=\prod\limits_{\ell\geqslant1}\left(1-\frac{s_\ell}{N(\mathfrak{b}_\ell)}\right)
\end{equation}
(cf.\ \cref{uw:bloki}).

Let $\vep>0$ and let $L$ be sufficiently large so that $\prod\limits_{\ell=1}^{L}\left(1-\frac{s_\ell}{N(\mathfrak{b}_\ell)}\right)<\prod\limits_{\ell\geqslant1}\left(1-\frac{s_\ell}{N(\mathfrak{b}_\ell)}\right)+\vep$. Then for each $n\in\N$ sufficiently large, by \cref{l:ensz}, we have
\begin{align*}
\frac{1}{|H_n|}&\log_2 \gamma^{\geq \underline{s}}(n) \leqslant\frac{1}{|H_n|} \log_2\gamma^{\geq \underline{s}_{L}}_L(n)\\
&\leqslant\prod_{\ell=1}^{L} \left(1-\frac{s_\ell}{N(\mathfrak{b}_\ell)}\right)+\vep + \frac{1}{|H_n|}\log_2 \prod_{\ell=1}^L {N(\mathfrak{b}_\ell) \choose s_\ell}\\
&\leqslant\prod_{\ell\geq1} \left(1-\frac{s_\ell}{N(\mathfrak{b}_\ell)}\right)+2\vep + \frac{1}{|H_n|}\log_2 \prod_{\ell=1}^L {N(\mathfrak{b}_\ell) \choose s_\ell}.
\end{align*}
Since $\vep>0$ can be arbitrarily small, we obtain
\begin{equation}\label{eq:teza1}
\limsup_{n\to\infty} \frac{1}{|H_n|} \log_2 \gamma^{\geq \underline{s}}(n)\leqslant\prod_{\ell\geqslant1} \left(1-\frac{s_\ell}{N(\mathfrak{b}_\ell)}\right).
\end{equation}

Fix $n\in\N$ and let $L(n)$ be as in~\eqref{eq:finit}. Then, by \cref{l:pom}, we have:
\begin{equation}
\begin{split}\label{eq:M1}
\frac{1}{|H_n|}\log_2 \gamma^{\geq\underline{s}}(n)&=\frac{1}{|H_n|}\log_2\gamma^{\geq \underline{s}_{L(n)}}_{L(n)}(n)\\
&\geqslant\frac{1}{|H_n|}\frac{1}{m^d}\log_2\gamma^{\geq \underline{s}_{L(n)}}_{L(n)}(nm).
\end{split}
\end{equation}
Moreover, it follows from \cref{l:ensz} that for all $m\geq M$ (where $M$ depends on $n$) we have
\begin{equation}\label{eq:M2}
\frac{1}{|H_{nm}|}\log_2\gamma^{\geq \underline{s}_{L(n)}}_{L(n)}(nm)\geqslant\prod_{l=1}^{L(n)} \ \left(1-\frac{s_\ell}{N(\mathfrak{b}_l)}\right)-\vep.
\end{equation}
Using~\eqref{eq:M1} and~\eqref{eq:M2}, we conclude that, for $m\geq M$,
\begin{equation*}
\begin{split}
\frac{1}{|H_n|}&\log_2 \gamma^{\geq \underline{s}}(n)\geq \frac{1}{|H_n|}\frac{|H_{nm}|}{m^d}\left(\prod_{\ell=1}^{L(n)} \ \left(1-\frac{s_\ell}{N(\mathfrak{b}_\ell)}\right)-\vep\right)\\
&=\frac{(2nm+1)^d}{m^d(2n+1)^d}\left(\prod_{\ell=1}^{L(n)} \ \left(1-\frac{s_\ell}{N(\mathfrak{b}_\ell)}\right)-\vep\right) \\
&\geqslant\frac{(2nm+1)^d}{m^d(2n+1)^d}\left(\prod_{\ell\geqslant1} \ \left(1-\frac{s_\ell}{N(\mathfrak{b}_\ell)}\right)-\vep\right).
\end{split}
\end{equation*}
Since $\vep>0$ can be arbitrarily small and $m$ arbitrarily large, we obtain
\begin{equation}\label{eq:teza2}
\liminf_{n\to\infty} \ \frac{1}{|H_n|} \log_2 \gamma^{\geq \underline{s}}(n)\geqslant\prod\limits_{\ell\geqslant1}\left(1-\frac{s_\ell}{N(\mathfrak{b}_\ell)}\right).
\end{equation}
It follows from~\eqref{eq:teza1} and~\eqref{eq:teza2} that~\eqref{eq:teza} indeed holds, and the proof is complete.
\end{proof}
\cref{thm:B} is clearly just a special case of \cref{pr:entropy}.

\subsection{Proof of \cref{thm:A}~\eqref{thmAii}}
The proof of \cref{thm:A}~\eqref{thmAii} consists of two main steps, which might themselves be of an interest:
\begin{Prop}\label{pr:maksym}
For any Erd\H{o}s set $\mathfrak{B}$ any measure of maximal entropy for $(X_\mathfrak{B},{(S_a)}_{a\in\OK})$ is concentrated on~$Y$.
\end{Prop}
\begin{Prop}\label{pr:gdzieeta}
For any Erd\H{o}s set $\mathfrak{B}$, we have $\nu_\eta(Y)=1$.
\end{Prop}

\begin{Remark}
An almost direct consequence of Proposition~\ref{pr:gdzieeta} is that $\eta\in Y$. Indeed, by~\eqref{rozkla}, $\eta\in Y_{\underline{s}}$ for some $\underline{s}=(s_\ell)_{\ell\geq1}$ such that $s_\ell\geq 1$, $\ell\geq1$. Moreover, by \cref{thm:A}~\eqref{thmAi} and Remark~\ref{uw:domk}, we obtain
$$
1=\nu_\eta(\overline{Y}_{\underline{s}})\leq \nu_\eta(Y_{\geq \underline{s}}).
$$
This contradicts Proposition~\ref{pr:gdzieeta}, since $Y_{\geq \underline{s}}\cap Y=\emptyset$.
\end{Remark}

One of the crucial tools will be the function $\theta\colon Y\to G$ given, for $y\in Y$, by
\begin{equation}\label{eq:deftheta}
\theta(y)=g \iff \text{supp }y \cap (\mathfrak{b}_\ell-g_\ell)=\emptyset \text{ for each }\ell\geq 1,
\end{equation}
where $g=(g_1,g_2,\dotsc)$. Notice that
\begin{equation}\label{lm:theta-equiv}
T_a\circ \theta=\theta\circ S_a\text{ for } a\in\OK.
\end{equation}
Indeed, take $a\in\mathcal{O}_K$, $y\in Y$ and let $g:=\theta(y)$. Then 
$$
T_a(\theta(y))=T_a(g)=(g_1+a,g_2+a,\dots).
$$
By the definition of $S_a$ we have $\text{supp } S_a y=\text{supp }y - a$. Hence, by the definition of~$\theta$, $(\text{supp }y-a) \cap (\mathfrak{b}_\ell-(g_\ell+a))=\emptyset \text{ for each }\ell\geq 1$.
This yields~\eqref{lm:theta-equiv}.

Before giving the proofs of Proposition~\ref{pr:maksym} and Proposition~\ref{pr:gdzieeta}, we show how to derive \cref{thm:A}~\eqref{thmAii} from them.
\begin{proof}[Proof of \cref{thm:A}~\eqref{thmAii}]
In view of Proposition~\ref{pr:gdzieeta}, we can consider $\varphi$ as a map whose codomain is $Y$, i.e.\ $\varphi\colon G\to Y$. Moreover, $\theta\colon Y\to\theta(Y)\subseteq G$. By \eqref{lm:varphi-eq} and \eqref{lm:theta-equiv}, we have 
$$
(\theta\circ\varphi)\circ T_a = T_a \circ (\theta\circ\varphi) \text{ for each }a\in \mathcal{O}_K.
$$
It follows by coalescence of ${(T_a)}_{a\in G}$ that $\theta\circ\varphi$ is a.e.\ invertible.\footnote{An automorphism $T$ of $\xbm$ is called {\em coalescent}~\cite{MR0230877} if each endomorphism commuting with $T$ is invertible. All ergodic automorphisms with purely discrete spectrum are coalescent. Both the definition and this fact extend to countable group actions.} In particular, $\varphi$ is 1-1 a.e., i.e.\ $\varphi$ yields the required isomorphism.
\end{proof}

\begin{proof}[Proof of Proposition~\ref{pr:maksym}]
Let $\nu$ be a measure of maximal entropy for $(X_\mathfrak{B}, {(S_a)}_{a\in\OK})$. 
By \cref{pr:entropy}, we have
\begin{equation}\label{eq:1}
h(X_\mathfrak{B},{(S_a)}_{a\in\OK},\nu)=h_{top}(X_\mathfrak{B},{(S_a)}_{a\in\OK})=\prod_{\ell\geq1}\left(1-\frac{1}{N(\mathfrak{b}_\ell)}\right).
\end{equation}

Suppose additionally that $\nu$ is ergodic. We claim that
\begin{equation}\label{eq:jedy}
\nu(Y_{\underline{s}})=1\text{ for some }\underline{s}=(s_\ell)_{\ell\geq1}.
\end{equation}
Indeed, let, for $\ell\geq 1$, $c_\ell \colon X_\mathfrak{B} \to \N$ be the measurable function given by
$$
c_\ell(x)=N(\mathfrak{b}_\ell)-D(\mathfrak{b}_\ell|\text{supp }x).
$$
Then, for any $\ell\geq 1$, we have
$
X_\mathfrak{B}=\bigsqcup_{k=1}^{N(\mathfrak{b}_\ell)}Y_k(\mathfrak{b}_\ell),
$
where $Y_k(\mathfrak{b}_\ell)=\{x\in X_\mathfrak{B} : c_\ell(x)=k\}$. Since $Y_k(\mathfrak{b}_\ell)$ are invariant and pairwise disjoint for a given $\ell\geq 1$, it follows by the ergodicity of $\nu$ that there exists a unique $1\leq s_\ell \leq N(\mathfrak{b}_\ell)$ such that $\nu(Y_{s_\ell}(\mathfrak{b}_\ell))=1$. This yields~\eqref{eq:jedy}. Since $Y_{\underline{s}}\subseteq Y_{\geq \underline{s}}$, it follows immediately that
$$
\nu(Y_{\geq \underline{s}})=1
$$
for the same choice of $\underline{s}$ as in~\eqref{eq:jedy}. By the variational principle and \cref{pr:entropy},
\begin{equation}\label{eq:4}
h(X_\mathfrak{B},{(S_a)}_{a\in\OK},\nu)
\leq h_{top}(Y_{\geq \underline{s}},{(S_a)}_{a\in\OK})=\prod_{\ell \geq 1}\left(1-\frac{s_\ell}{N(\mathfrak{b}_\ell)}\right).
\end{equation}
Comparing \eqref{eq:1} and \eqref{eq:4}, we conclude that
$$
\prod_{\ell\geq1}\left(1-\frac{1}{N(\mathfrak{b}_\ell)}\right) \leq \prod_{\ell\geq1}\left(1-\frac{s_\ell}{N(\mathfrak{b}_\ell)}\right).
$$
This is however true only if $s_\ell=1$ for all $\ell\geq 1$, whence indeed $\nu(Y)=1$.

If $\nu$ is not ergodic, we write its ergodic decomposition. It follows by \eqref{uw:wypu} that almost every measure in this decomposition is also of maximal entropy, whence it is concentrated on $Y$. Thus also $\nu(Y)=1$.
\end{proof}

\begin{proof}[Proof of Proposition~\ref{pr:gdzieeta}]
We will show that
\begin{equation}\label{nm:1}
\nu_\eta(\varphi(\theta(Y)))=1
\end{equation}
and
\begin{equation}\label{nm:2}
\varphi(\theta(Y))\subseteq Y,
\end{equation}
and the assertion will follow immediately. Let $\nu$ be an invariant measure concentrated on $Y$ (in view of Proposition~\ref{pr:maksym}, we can take for $\nu$ any measure of maximal entropy).

For~\eqref{nm:1}, notice first that~\eqref{lm:theta-equiv} and the unique ergodicity of the rotation on $\mathbb{G}$ yield $\theta_\ast(\nu)=\PP$. Therefore and by Proposition~\ref{pr:maksym},
\begin{equation*}
\nu_\eta(\varphi(\theta(Y)))=\PP(\varphi^{-1}(\varphi(\theta(Y))))\geq \PP(\theta(Y))
=\theta_\ast\nu(\theta(Y))=\nu(\theta^{-1}(\theta(Y)))\geq \nu(Y)=1,
\end{equation*}
i.e.\ \eqref{nm:1} indeed holds. We will now show~\eqref{nm:2}, by proving
\begin{equation}\label{lm:varphi-max}
y\leq \varphi(\theta(y)) \text{ for each }y\in Y.
\end{equation}
Take $y\in Y$ and suppose that $\varphi(\theta(y))(a)=0$. By the definition of $\varphi$, this means that for some $\ell\geq 1$ we have 
$$
\theta(y)_\ell+a \equiv 0 \bmod \mathfrak{b}_\ell.
$$
In other words,  $\theta(y)_\ell+a\in\mathfrak{b}_\ell$, i.e.\ $a\in \mathfrak{b}_\ell-\theta(y)_\ell$. It follows from~\eqref{eq:deftheta} that $y(a)=0$. This yields~\eqref{lm:varphi-max} and the proof is complete.
\end{proof}

\subsection{Proof of \cref{thm:E}}
For $x\in X_\mathfrak{B}$ and $\ell\geqslant 1$ let
$$
F_\ell(x):=\{c\bmod \mathfrak{b}_\ell\colon x|_{-c+\mathfrak{b}_\ell}\equiv 0\}.
$$
Then $F:=(F_1,F_2,\dots)$ defines a multivalued function $F\colon X_{\mathfrak{B}}\to G$. Let
$$
A:=cl(\text{Graph}(F)).
$$
We claim that
\begin{enumerate}[(i)]
\item\label{i'}
$(S_a\times T_a)(A)=A$ for each $a\in\mathcal{O}_K$,
\item\label{ii'}
$\pi_{X_{\mathfrak{B}}}(A)=X_{\mathfrak{B}}\text{ and }\pi_{G}(A)=G$, where $\pi_{X_\mathfrak{B}}$ and $\pi_G$ stand for the corresponding projections,
\item\label{iii'}
$A\neq X_{\mathfrak{B}}\times G$.
\end{enumerate}
In order to prove~\eqref{i'}, it suffices to show $F\circ S_a= T_a\circ F$. Indeed, for $(x,\omega)\in X_{\eta}\times G$, we have
\begin{align*}
\omega\in F(S_a x)&\iff S_a x|_{-\omega_\ell+\mathfrak{b}_\ell}\equiv 0  \text{ for all }\ell\geq 1\\
&\iff x|_{a-\omega_\ell+\mathfrak{b}_\ell}\equiv 0 \text{ for all }\ell\geq 1\\\
&\iff T_{-a}\omega\in F(x) \iff \omega \in T_a(F(x)).
\end{align*}
Clearly, $\pi_{X_{\mathfrak{B}}}(A)=X_{\mathfrak{B}}$. Moreover, we have $F(\mathbf{0})=G$. This yields~\eqref{ii'}. For the last part of our claim consider $x\in X_\mathfrak{B}$ such that $x(0)=1$ and $x(a)=0$ for $a\neq 0$.
Notice that for all $\ell\geqslant 1$ we have $0\not \in F_\ell(x)$, whence
\[
F(x)\subseteq \prod_{\ell\geqslant 1}\left((\OK/\mathfrak{b}_\ell)\setminus \{0\}\right).
\]
Moreover, if $y\in X_{\mathfrak{B}}$ is such that $d(x,y)$ is small enough then $y(0)=x(0)=1$, which yields
\[
F(y)\subseteq \prod_{\ell\geqslant 1}\left((\OK/\mathfrak{b}_\ell)\setminus \{0\}\right).
\]
It follows that $(x,\omega)\not \in A$, whenever $\omega_\ell=0$ for some $\ell\geq 1$. This completes the proof of \cref{thm:E}.

\section{From $\mathfrak{B}$-free integers to $\mathscr{B}$-free lattice points}\label{a:a}
Clearly, \eqref{dwojka} is a special case of \eqref{czworka}. Moreover, \eqref{trojka} is a special case of~\eqref{nasz:setting} since
\[
\sum_{\mathfrak{p}\in\mathfrak{P}}\frac{1}{N(\mathfrak{p}^k)}\leq \sum_{\mathfrak{a} \neq \{0\}}\frac{1}{N(\mathfrak{a})^k}=\zeta_K(k)<\infty \text{ for }k\geq 2
\]
and in a Dedekind domain any two prime ideals $\mathfrak{p}\neq \mathfrak{q}$ are coprime. Our goal is to show now that Sarnak's program \eqref{s:A}-\eqref{s:E} in case \eqref{czworka} can be easily obtained using the results in setting~\eqref{nasz:setting} presented in \cref{se:results}. 
Let $K$ be a number field of degree $d$. Fix a lattice $\Lambda$ in $\R^d$. Let 
\[
j\colon \Lambda \to \OK\text{ be a group isomorphism}
\]
(cf.~\eqref{fo:17}). Consider two actions by translations: ${(S_a)}_{a\in\OK}$ on $\{0,1\}^{\OK}$ and $({S_\mathbf{n})}_{\mathbf{n}\in\Lambda}$ on $\{0,1\}^\Lambda$ (see~\eqref{eq:translations} for the formulas).
\begin{Remark}\label{m:1}
Notice that ${(S_a)}_{a\in\OK}$ on $\{0,1\}^{\OK}$ and $({S_\mathbf{n})}_{\mathbf{n}\in\Lambda}$ on $\{0,1\}^\Lambda$ are two different representations of the same (topological) dynamical system. Indeed, let $S_J\colon \{0,1\}^{\Lambda}\to\{0,1\}^{\OK}$ be given by
\[
S_J(x)(a):=x(j^{-1}(a))\text{ for }a\in\OK.
\]
Then, for each $\mathbf{n}\in\Lambda$, we have 
\(
S_{j(\mathbf{n})}=S_J\circ S_{\mathbf{n}}\circ S_J^{-1}.
\)
\end{Remark}
Fix an infinite pairwise coprime set $\mathscr{B}:=\{b_\ell :\ell\geq 1\}\subseteq\N$ satisfying $\sum_{\ell\geq 1}\frac{1}{b_\ell^d}<\infty$. Then each $L_\ell:=b_\ell \Lambda$ is a sublattice of $\Lambda$ and each $\mathfrak{b}_\ell:=j(L_\ell)$ is an ideal in $\OK$. Since $j$ is a group isomorphism,
$\mathfrak{B}$ is Erd\H{o}s and
the set of $\mathscr{B}$-free lattice points in $\Lambda$ defined as
$
\mathcal{F}_\mathscr{B}=\mathcal{F}_\mathscr{B}(\Lambda):=\Lambda\setminus \bigcup_{\ell\geq 1}b_\ell\Lambda
$
satisfies
\begin{equation}\label{m:2}
j(\mathcal{F}_\mathscr{B})=\mathcal{F}_\mathfrak{B},
\end{equation}
where $\mathcal{F}_\mathfrak{B}$ is the corresponding set of $\mathfrak{B}$-free integers (defined as in~\eqref{Z:1}).
Moreover, any residue class modulo $j(L_\ell)$ corresponds to a unique residue class modulo $L_\ell$. Hence Theorem \ref{thm:C} implies part \eqref{s:C} of Sarnak's program in setting \eqref{czworka}.

Let $H:=\prod_{\ell\geq 1}\Lambda / b_\ell\Lambda$ and let $\widetilde{\PP}$ stand for Haar measure on $H$ (cf.\ \eqref{eq:grupa}). Notice that this group is isomorphic to $G$ via the map $J\colon H\to G$ given by
\[
J(h)=(j(h_1),j(h_2),\dots)\text{ for }h=(h_1,h_2,\dots).
\]
On $H$ we have a natural $\Lambda$-action ${(T_\mathbf{n})}_{\mathbf{n}\in\Lambda}$:
\[
T_\mathbf{n}(h)=(h_1+\mathbf{n},h_2+\mathbf{n},\dotsc)\text{ for }h=(h_1,h_2,\dots)\in H
\]
(cf.\ \eqref{eq:rota}).\footnote{Notice that both $J$ and ${(T_\mathbf{n})}_{\mathbf{n}\in\Lambda}$ are well-defined.}
\begin{Remark}\label{m:3}
Notice that ${(T_a)}_{a\in \OK}\colon G\to G$ and ${(T_\mathbf{n})}_{\mathbf{n}\in\Lambda}\colon H\to H$ are two different representations of the same (algebraic and topological) dynamical system. Indeed, we have 
\begin{equation}
\text{$T_{j(\mathbf{n})}=J\circ T_{\mathbf{n}}\circ J^{-1}$ for each $\mathbf{n}\in\Lambda$.}
\end{equation}
\end{Remark}
Define $\widetilde{\varphi}\colon H\to \{0,1\}^\Lambda$ in a similar way as $\varphi$ in~\eqref{eq:wzor}:
\[
\widetilde{\varphi}(h)(\mathbf{n})=1 \iff h_\ell + \mathbf{n} \not\in L_\ell\text{ for each }\ell\geq 1.
\]
\begin{Remark}\label{m:4}
Notice that $\widetilde{\varphi}$ is the function which ``corresponds'' to $\varphi$ when we take into account isomorphisms from Remark~\ref{m:1} and Remark~\ref{m:3}. Indeed, we have $\widetilde{\varphi}=S_J^{-1}\circ \varphi\circ J$. It follows that
\[
\widetilde{\nu}_\eta:={\widetilde{\varphi}}_\ast(\widetilde{\PP})=(S_J^{-1}\circ \varphi \circ J)_\ast(\widetilde{\PP})=(S_J^{-1})_\ast(\nu_\eta).
\]
\end{Remark}
Since the topological conjugacy preserves genericity, the value of topological entropy, the number of minimal sets, proximality, triviality of the maximal equicontinuous factor and non-trivial topological joinings,   
by Theorems \ref{thm:A}, \ref{thm:B}, \ref{thm:E}, the formula \eqref{m:2} and Remarks~\ref{m:1}, \ref{m:3} and \ref{m:4}, we obtain that
parts \eqref{s:A}, \eqref{s:B}, \eqref{s:E} of Sarnak's program in setting \eqref{czworka} are covered. 
\section*{Acknowledgments}
We would like to thank M.~Lema\'{n}czyk and I.~Vinogradov for the helpful comments on the preliminary version of this paper. We also thank T.~Downarowicz and B.~Kami\'{n}ski for their help with references to the entropy of $\Z^d$-actions.

FA was supported by the Deutsche Forschungsgemeinschaft (DFG, German Research Foundation) - Project-ID 491392403 - TRR 358 (project A2).

\bigskip
\footnotesize
\noindent
Francisco Ara\'{u}jo\\
\textsc{Institute of Mathematics, Paderborn University, Warburger Str. 100, 33098 Paderborn, Germany}\par\nopagebreak
\noindent
\textit{E-mail address:} \texttt{faraujo@math.uni-paderborn.de}

\medskip

\noindent
Aurelia Dymek\\
\textsc{Faculty of Mathematics and Computer Science, Nicolaus Copernicus University, Chopina 12/18, 87-100 Toru\'{n}, Poland}\par\nopagebreak
\noindent
\textit{E-mail address:} \texttt{aurbart@mat.umk.pl}

\medskip

\noindent
Joanna Ku\l aga-Przymus\\
\textsc{Faculty of Mathematics and Computer Science, Nicolaus Copernicus University, Chopina 12/18, 87-100 Toru\'{n}, Poland}\par\nopagebreak
\noindent
\textit{E-mail address:} \texttt{joanna.kulaga@gmail.com}

\end{document}